\documentclass[a4paper, english, 10pt]{amsart}

\usepackage{amsthm,amssymb,url,hyperref}
\usepackage{fullpage}
\usepackage{stmaryrd}
\usepackage[all]{xy}
\usepackage{color}
\usepackage[all]{xy}
\usepackage{tikz}
\usepackage{hyperref}

\theoremstyle{plain}
\newtheorem{theorem}{Theorem}[section]

\newtheorem{corollary}[theorem]{Corollary}

\newtheorem{lemma}[theorem]{Lemma}

\newtheorem{proposition}[theorem]{Proposition}

\theoremstyle{definition}

\newtheorem{example}[theorem]{Example}
\newtheorem{remark}[theorem]{Remark}

\def\!#1{!#1}

\def\ker#1{\mathrm{ker}(#1)}

\def\aut#1{\mathrm{Aut}(#1)}
\def\Adj{\mathrm{Adj}}

\def\aff#1{\mathrm{Aff}#1}
\def\Aff#1{\mathrm{Aff}#1}
\def\lmlt{\mathrm{LMlt}}
\def\dis{\mathrm{Dis}}
\def\rad{\mathrm{rad}}

\def\sym{\mathrm{Sym}}
\def\Sym{\mathrm{Sym}}

\newcommand{\core}{\operatorname{\mathsf{Core}}}
\newcommand{\Conj}{\operatorname{\mathsf{Conj}}}

\def\comment#1{{\color{red} #1}}

\def\setof#1#2{\{#1\, : \,#2\}}

\def\Z{\mathbb Z}

\def\ldiv{\backslash}

\def\cg#1{\equiv_\alpha}
\newcommand*\xbar[1]{%
   \hbox{%
     \vbox{%
       \hrule height 0.5pt 
       \kern0.5ex
       \hbox{%
         \kern-0.1em
         \ensuremath{#1}%
         \kern-0.1em
       }%
     }%
   }%
} 

\title{On simply connected quandles}

\author{M. Bonatto}

\address[M. Bonatto]{}

\email{marco.bonatto.87@gmail.com}

\begin{document}

\begin{abstract}
    In this paper we provide an alternative characterization of finite simply connected quandles involving only cocycles with values in abelian groups of prime size. As a corollary of such a characterization and the classification of connected quandles of size $p^2$ and $p^3$ we obtain a classification of simply connected quandles of size $p^2$ (already obtained with a different method in \cite{VV}) and $p^3$ for $p>3$ using a method that works for quandles of size $p^n$ for arbitrary $n$. We also classify the simply connected quandles within two subclasses of finite involutory quandles: nilpotent latin quandles and core quandles.
\end{abstract}
\maketitle

\section*{Introduction}

Quandles are algebraic structures introduced in \cite{J} and \cite{Matveev} in order to provide algebraic invariants of knots. Quandles have also been studied in connection with Hopf algebras and in the framework of the Yang-Baxter equation \cite{AG,ESS, EGS}. They have also been investigated from a purely algebraic viewpoint by several authors, using techniques from group theory, module theory and universal algebra. The study of quandle cohomology has been started to offer new topological invariants for knots and links too \cite{CEGS}. Nonetheless, it can also be used in order to construct quandle extensions with special properties, known under the names of {\it quandle covers} (or {\it covering maps}) for quandles. Quandle covers have been investigated in the seminal paper \cite{Eisermann} in analogy with coverings of topological spaces by using a categorical approach. They have been further studied in \cite{covering_paper} from a universal algebraic viewpoint. Quandles admitting only trivial covers are said to be {\it simply connected} and they have been defined in \cite{Eisermann} (again in analogy with simply connected spaces). A categorical characterization of simply connected quandles have also been obtained in the same paper and it involves the so-called {\it enveloping group} of the quandle that plays the same role of the fundamental group of a topological space. A slightly different characterization was also given in \cite{covering_paper} and it shows in particular that simply connected quandles are {\it principal}, i.e. they can be obtained using a construction over groups.  


In this paper we focus on quandle covers and simply connected quandles by using universal algebraic and group theoretical techniques. For instance we prove that for finite connected quandles the notions of {\it trivial cover} and {\it trivial cocycle} coincide (see Proposition \ref{trivial for connected} and Corollary \ref{trivial for connected finite}). We were also able to answer to the open question asked in \cite[Remark 9.7]{MeAndPetr} for connected quandles accordingly (see Proposition \ref{simply then all groups}). We also provide an alternatively characterization of finite simply connected quandles in Theorem \ref{H2 abelian}, proving that we can study just quandle cocycles with values in abelian groups of prime size. 

We specialize Theorem \ref{H2 abelian} to quandles of prime power size in Theorem \ref{simply pn} and we employ the theorem to classify simply connected quandles of size $p^2$ and $p^3$ for $p>3$ (the former was already obtained in \cite{VV}). This result relies on the classification of connected quandles of size $p^2$ and $p^3$ obtained in \cite{Hou} and \cite{GB} and on the classification of groups of size $p^4$ and their automorphisms given in \cite{tedesco}. We actually describe a (purely group theoretical) algorithm that can tell if a quandle of size $p^n$ is simply connected or not, provided the description of all groups of size $p^{n+1}$ and their automorphisms. The algorithm that allows to construct groups of size $p^n$ and their automorphisms provided the groups of size $p^{n-1}$ is a classical result and it is described in details in \cite{tedesco}. Combining both algorithms it is possible to push the classification of simply connected quandles of size $p^n$ for arbitrary $n$. Moreover, principal nilpotent quandles (in the sense of \cite{comm} and \cite{CP}) are direct products of quandles of prime power order and in the finite case one of such quandles is simply connected if and only if its $p$-components are simply connected (see Theorem \ref{simply nilpotent iff}). So the algorithm can be actually used to determine wether a finite nilpotent quandle is simply connected or not.

For latin quandles we can also use that the property of being simply connected is stable under homomorphic images (see Proposition \ref{simply factor}). Indeed, given a quandle, if it has a factor that is not simply connected, then it is not simply connected. In this way we can shorten the list of quandles to be processed by the algorithm.

Concerning quandle cocycles, we also define a special family of constant quandle cocycles named {\it split cocycles} by using quandle homomorphisms, providing a concrete way to build cocycles and also cocycle invariants consequently (see Lemmas \ref{twisted} and \ref{cocycle by rho core}).

Using the classification of simply connected quandles of size $p^2$ and split cocycles we obtain also the classification simply connected quandles within two subclasses of finite involutory quandles: nilpotent latin quandles (Theorem \ref{involutory simply}) and core quandles (Corollary \ref{simply cores}).

The paper is organized as follows: in Section \ref{preliminary} we include all the basic results on left quasigroups and quandles we need in the paper, in particular towards the Cayley kernel and quandle covers. In Section \ref{cohomology} we talk about constant quandle cocycle and we characterize cocycles cohomologous to the trivial cocycles in terms of congruences. We also explain how abelian cocycles can be used to build covers of quandles under some additional assumptions. In Section \ref{simply} we focus on simply connected quandles. We characterize finite simply connected quandles in general and the nilpotent ones in particular. Moreover, we obtain the classification of simply connected quandles of size $p^2$ and $p^3$ for every prime $p>3$. In the last section we turn our attention to simply connected involutory quandles.  Appendix \ref{appendix} collects some group theoretical results we need for the classification results.

\textbf{Notation:} let $G$ be a group and $g\in G$. We denote by $\widehat{g}$ the inner automorphism of $G$ with respect to $g$. The derived subgroup of $G$ will be denoted by $\gamma_1(G)$ and the Frattini subgroup of $G$ by $\Phi(G)$. If $f\in \aut{G}$ and $N$ is a normal subgroup such that $f(N)=N$, then $f$ induces an automorphism on $G/N$ defined as $f_N(xN)=f(x)N$ for every $x\in G$.

Let $G$ be acting on a set $X$. We denote by $G_x$ the stabilizer of $x\in X$ and by $x^G$ the orbit of $x\in X$ under the action of $G$.

\section{Preliminary results}   \label{preliminary}
\subsection{Left quasigroups and quandles}

A {\it left quasigroup} is a set $Q$ endowed with a pair of binary operations $\{*,\backslash\}$ such that the identities
\begin{align}\label{LQG}
    x*(x\backslash y)\approx y\approx x\backslash (x*y)
\end{align}
hold. We usually denote $*$ just by juxtaposition. We can define the maps
$$L_x:y\mapsto x*y,\quad R_x:y\mapsto y*x$$
for every $x\in Q$. Note that, according to \eqref{LQG}, $L_x$ is a permutation of $Q$ with inverse $L_x^{-1}:y\mapsto x\backslash y$. So we can define the {\it left multiplication group} of $Q$ as $\lmlt(Q)=\langle L_x,\, x\in Q\rangle$.

A left quasigroup is said to be:
\begin{itemize}
    \item[(i)] {\it connected}, if $\lmlt(Q)$ is transitive over $Q$;
    \item[(ii)] {\it latin} if $R_x$ is bijective for every $x\in Q$;
    \item[(iii)] {\it involutory} if $x*(x*y)=y$ for every $x,y\in Q$. 
\end{itemize}
Latin left quasigroups are basically {\it quasigroups} for which we do not consider right division as a basic operation (defined as $x/y=R_y^{-1}(x)$ for every $x,y\in Q$). If the underlying set is infinite considering the right division as a basic operation or not might effect congruences and subalgebras. Note that subalgebras, homomorphic images and direct products (with respect to the set of basic operations $\{*,\backslash\}$) of finite latin left quasigroups are latin. See \cite{Q1, Q2} for further details on quasigroups

A {\it rack} is a left quasigroup $(Q,*,\backslash)$ such that the identity
\begin{align*}
    x*(y*z)\approx(x*y)*(x*z)
\end{align*}
holds. Idempontent racks (i.e. racks such that $x*x\approx x$ holds) are called {\it quandles}. 

\begin{example}\label{examples}\text{ }
\begin{itemize}
    \item[(i)]     Projection left quasigroups, i.e. left quasigroups for which $x*y=y$ for every $x,y\in Q$ are quandles. We call them {\it projection quandles}. The only connected projection quandle has size $1$.
\item[(ii)] Let $G$ be a group and $H\subseteq G$ a set closed under conjugation. Then $H$ together with the operation $x*y=xyx^{-1}$ is a quandle, denoted by $\Conj(H)$. Note that $\Conj$ is a functor from the category of groups to the category of quandles.

\item[(iii)] Let $G$ be a group, $f\in \aut{G}$ and $H\leq Fix(f)=\setof{g\in G}{f(g)=g}$ we can define the {\it coset quandle} over $G/H$ with operation $*$ defined by setting
$$xH*yH=xf(x^{-1}y)H$$
for every $x,y\in G$. We denote such quandle as $Q=\mathcal{Q}(G,H,f)$. If $H=1$ we say that $Q$ is {\it principal} and we denote it just by $\mathcal{Q}(G,f)$. If in addition $G$ is abelian we denote $Q=\mathcal{Q}(G,f)$ by $\aff(G,f)$ and we say that $Q$ is {\it affine} over $G$.

\item[(iv)] Let $G$ be a group and $f\in \aut{G}$. Then $G$ with the operation $x*y=xf(yx^{-1})$ is a quandle, called {\it twisted conjugation quandle} and denoted by $\Conj_f(G)$. Note that if $G$ is abelian such quandles are affine and that $\Conj_1(G)=\Conj(G)$.

\item[(v)] Let $G$ be a group. Then $G$ with the operation $x*y=xy^{-1}x$ is an involutory quandle, called {\it core} quandle and denoted by $\core(G)$. Note that $\core$ is a functor from the category of groups to the category of involutory quandles \cite{Bardakov}.
    \end{itemize}

\end{example}

A congruence of a left quasigroup $Q$ is an equivalence relation compatible with the operations $\{*,\backslash\}$. Namely, a congruence is an equivalence relation $\alpha$ such that $$(x*y)\, \alpha\, (z*t)\, \text{ and }\, (x\backslash y)\, \alpha\, (z\backslash t)$$ provided $x\,\alpha\, z$ and $y\, \alpha\, t$. We denote the lattice of congruences of $Q$ as $Con(Q)$ with top element $1_Q=Q\times Q$ and bottom element $0_Q=\setof{(x,x)}{x\in Q}$. Congruences and surjective homomorphisms of left quasigroups are essentially the same thing. Given $\alpha\in Con(Q)$ we can define a left quasigroup structure on the set $Q/\alpha$ with operations $[x]*[y]=[x*y]$ and $[x]\backslash [y]=[x\backslash y]$ for every $x,y\in Q$. The canonical map $x\mapsto [x]$ is a left quasigroup homomorphism. On the other hand, given a surjective homomorphism $f:Q\longrightarrow Q'$ the relation $\ker{f}=\setof{(x,y)\in Q\times Q}{f(x)=f(y)}$ is a congruence and $Q'\cong Q/\ker{f}$. 

Given $\alpha\in Con(Q)$, the map 
\begin{equation}
    \pi_\alpha:\lmlt(Q)\longrightarrow \lmlt(Q/\alpha),\quad L_x\mapsto L_{[x]}
\end{equation}
can be extended to a well defined surjective group homomorphism with kernel denoted by $\lmlt^\alpha$ (see \cite{AG}). Moreover we have that
\begin{align}\label{pi(h)}
[h(x)]_\alpha=\pi_\alpha(h)([x]_\alpha)    
\end{align}
for every $x\in Q$ and $h\in \lmlt(Q)$.

For racks the interplay between congruences and normal subgroups of the left multiplication group is pretty strong and it has been investigated in \cite{CP} and \cite{semimedial}. Indeed, if $\alpha$ is a congruence of a rack $Q$ we can define the {\it displacement group relative to $\alpha$} as 
\begin{align}\label{dis_alpha}
    \dis_\alpha=\langle L_x L_y^{-1},\, x\,\alpha\, y\rangle.
\end{align}
Such subgroups are normal subgroups of $\lmlt(Q)$ \cite[Section 3.1]{CP}. In particular, we denote $\dis_{1_Q}$ just as $\dis(Q)$ and we call it {\it the displacement group} of $Q$. We can also define $\dis^\alpha=\lmlt^\alpha\cap \dis(Q)$ and $\dis(Q)_{[x]}=\setof{h\in \dis(Q)}{h(x)\,\alpha\, x}$ for every $x\in Q$. Then, for $x\in Q$ we have that
$$\dis(Q)_x\leq\dis(Q)_{[x]_\alpha}\quad \text{ and }\quad \dis_\alpha\leq \dis^\alpha=\bigcap_{x\in Q}\dis(Q)_{[x]_\alpha} .$$
The coset quandle construction can be used to represent connected quandles over their displacement groups.

\begin{proposition}\cite{J}\label{coset reps}
    Let $Q$ be a connected quandle and $x\in Q$. Then $$Q\cong \mathcal{Q}(\dis(Q),\dis(Q)_x,\widehat{L}_x).$$
\end{proposition}

We can also provide a coset representation of factors as follows. The second statement was not included in the original paper but can be verified by a direct computation.
\begin{proposition}\label{coset reps quotient}\cite[Lemma 1.4]{GB}
    Let $Q$ be a connected quandle, $x\in Q$ and $\alpha\in Con(Q)$. Then $$Q/\alpha\cong \mathcal{Q}(\dis(Q)/\dis^\alpha,\dis(Q)_{[x]}/\dis^\alpha,(\widehat{L}_x)_{\dis^\alpha})$$
    and the map 
    \begin{align*}
    \pi_\alpha:\mathcal{Q}(\dis(Q),\dis(Q)_x,\widehat{L_x})&\longrightarrow  \mathcal{Q}(\dis(Q)/\dis^\alpha,\dis(Q)_{[x]}/\dis^\alpha,(\widehat{L}_x)_{\dis^\alpha})  \\ 
    h\dis(Q)_x& \mapsto \pi_\alpha(h)\dis(Q)_{[x]}/\dis^\alpha\end{align*}

    is a morphism of quandle.
\end{proposition}

The subgroups defined in \eqref{dis_alpha} are revelant in order to describe the properties of congruences, in particular towards the notions of {\tt abelianness} and {\it centrality} of congruences and the related concepts of {\it solvability} and {\it nilpotence} of algebras that have been developed in \cite{comm} in order to generalize the usual definitions for groups. In \cite{CP} we adapted to racks and quandles the commutator theory in the sense of \cite{comm}. The abelianness and centrality of congruences are completely determined by the properties of the correspondent subgroups defined in \eqref{dis_alpha}. Indeed a congruence $\alpha$ of a quandle $Q$ is {\it abelian} (resp. {\it central}) if $\dis_\alpha$ is abelian (resp. central in $\dis(Q)$) and $(\dis_\alpha)_x=(\dis_\alpha)_y$ (resp. $\dis(Q)_x=\dis(Q)_y$) provided $x\,\alpha\, y$ \cite[Theorem 1.1]{CP}. 

A quandle is {\it solvable} (resp. {\it nilpotent}) of length $n$ (we also say that $Q$ is $n$-step solvable/nilpotent) if there exists a chain of congruences
$$0_Q= \alpha_0\leq \alpha_1\leq\ldots \leq\alpha_{n-1}\leq \alpha_n=1_Q$$
such that $\alpha_{i+1}/\alpha_i$ is abelian (resp. central) in $Q/\alpha_i$.  It turned out that solvable (resp. nilpotent) quandles are exactly the quandles with solvable (resp. nilpotent) displacement group (see \cite[Theorem 1.2]{CP}). Solvable quandles of length $1$ are called {\it abelian} and we have that a quandle $Q$ is abelian if and only if $\dis(Q)$ is abelian and semiregular \cite[Theorem 2.2]{abelian_quandles}. In particular, using Proposition \ref{coset reps} we have that a connected abelian quandle is affine over its displacement group \cite{Medial}.

On the other hand, if $N$ is a normal subgroup of $\lmlt(Q)$ then
\begin{align*}
    \mathcal{O}_N=\setof{(x,h(x))\in Q^2}{x\in Q,\, h\in N},
\end{align*}
is a congruences of $Q$ \cite[Lemma 2.6]{CP}. In particular, $Q/\mathcal{O}_{\lmlt(Q)}$ is a projection quandle and the orbits with respect to $\lmlt(Q)$ are subalgebras (indeed if $y=h(x)$ for some $h\in\lmlt(Q)$ then $L_y^{\pm 1}(x)=L_{h(x)}^{\pm 1}(x)=h L_x^{-1} h^{-1}(x)\in x^{\lmlt(Q)}$).

As for groups we can identify the smallest congruence with abelian factor (see \cite{comm}). Let us denote by $\gamma_Q$ such a congruence for a quandle $Q$. Namely we have that $Q/\alpha$ is an abelian quandle if and only if $\gamma_Q\leq \alpha$. Moreover, the property of being connected is determined by the factor with respect to $\gamma_Q$.

\begin{proposition}\label{connected by gamma}
    Let $Q$ be a quandle. The following are equivalent:
    \begin{enumerate}
        \item[(i)] $Q$ is connected.
        \item[(ii)] $Q/\gamma_Q$ is connected.
    \end{enumerate}
\end{proposition}

\begin{proof}
    (i) $\Rightarrow$ (ii) Homomorphic images of connected quandles are connected.

    (ii) $\Rightarrow$ (i) Let $\alpha=\mathcal{O}_{\lmlt(Q)}$. Then $Q/\alpha$ is a projection quandle, so $\dis(Q/\alpha)=1$ and in particular it is abelian and semiregular. Therefore $Q/\alpha$ is abelian and so $\gamma_Q\leq \alpha$. Hence we have an onto quandle morphism $Q/\gamma_Q\longrightarrow Q/\alpha$. Then $Q/\alpha$ is connected and projection. So, $|Q/\alpha|=1$ (see Example \ref{examples}). Thus $\alpha=1_Q$, i.e. $Q$ is connected.
\end{proof}

For a finite quandle $Q$, the congruence $\gamma_Q$ can be realized by using the orbits of the derived subgroup of $\dis(Q)$. 

\begin{proposition}[\cite{GB}, Proposition 1.6] \label{gamma for quandles}
    Let $Q$ be a finite connected quandle and $G=\dis(Q)$. Then $\gamma_Q=\mathcal{O}_{\gamma_1(G)}$ and $\dis^{\gamma_Q}=\gamma_1(G)$.
\end{proposition}



The orbit decomposition with respect to the action of the left multiplication group is a congruence also for left quasigroups (see \cite[Corollary 1.9]{semimedial}). We conclude this section by looking at how connected factors and orbits with respect to the left multiplication group of left quasigroups are related.
\begin{lemma}\label{block onto}
Let $Q$ be a left quasigroup and $\alpha,\beta$ be congruences of $Q$. If $\alpha\circ \beta=1_Q$ then the map 
$$[x]_\alpha\longrightarrow Q/\beta,\quad y\mapsto [y]_\beta$$ 
is onto for every $x\in Q$.
\end{lemma}

\begin{proof}
Since $\alpha\circ \beta=1_Q$, for every $x,y\in Q$ there exists $z\in Q$ such that $x\,\alpha\, z\,\beta\, y$. Therefore $z\in [x]_\alpha$ and $[z]_\beta=[y]_\beta$. Thus the canonical map $z\mapsto [z]_\beta$ restricted to $[x]_\alpha$ is onto.	
\end{proof}

\begin{lemma}\label{connected factor}
Let $Q$ be a left quasigroup and $\alpha\in Con(Q)$. The following are equivalent:
\begin{itemize}
\item[(i)] $Q/\alpha$ is connected.
\item[(ii)] $\alpha\circ \mathcal{O}_{\lmlt(Q)}= \mathcal{O}_{\lmlt(Q)}\circ \alpha=1_Q$.
\end{itemize}
\end{lemma}

\begin{proof}

Note that $[x]\,\mathcal{O}_{\lmlt(Q/\alpha)}\, [y]$ if and only if $[x]_\alpha=\pi_{\alpha}(h)([y]_\alpha)=[h(y)]_\alpha$, i.e. $x \, \alpha \, h(y)\, \mathcal{O}_{\lmlt(Q)}\,y$ for some $h\in \lmlt(Q)$. Therefore $[x]\,\mathcal{O}_{\lmlt(Q/\alpha)}\, [y]$ if and only if $x\,\alpha\circ \mathcal{O}_{\lmlt(Q)}\, y$. Thus, $\alpha\circ \mathcal{O}_{\lmlt(Q)} =1_Q$ if and only if $Q/\alpha$ is connected.
\end{proof}

Combining Lemma \ref{block onto} and Lemma \ref{connected factor} we have the following result.
\begin{corollary}\label{orbits are onto}
    Let $Q$ be a left quasigroup, $\alpha\in Con(Q)$, $x\in Q$ and $Q'=x^{\lmlt(Q)}$. If $Q/\alpha$ is connected then the canonical map $x\mapsto [x]_\alpha$ restricted to $Q'$ is onto.
\end{corollary}

\subsection{The Cayley kernel and Rack covers}

Let $Q$ be a rack. The {\it Cayley kernel} of $Q$ is the equivalence relation $\lambda_Q$ defined by setting $x\,\lambda_Q\, y$ if and only if $L_x=L_y$. Such a relation is a congruence for racks. We say that $Q$ is {\it faithful} if $\lambda_Q=0_Q$ and {\it superfaithful} if every subalgebra of $Q$ is faithful. If $\alpha$ is a congruence of $Q$ and $\alpha\leq \lambda_Q$ we say that $Q$ is a {\it cover} of $Q/\alpha$ (equivalently $\dis_\alpha=1$). If $p:Q\longrightarrow Q'$ is a surjective quandle morphism such that $\ker{p}\leq \lambda_Q$ we say that $p$ is a {\it covering map}. Congruences below the Cayley kernel have a universal algebraic understanding as {\it strongly abelian} congruences in the sense of \cite{comm} as explained in \cite[Section 6]{covering_paper}. 

\begin{lemma}\label{orbits are connected}
Let $Q$ be a rack and $\alpha\leq \lambda_Q$.  If $Q/\alpha$ is connected then $x^{\lmlt(Q)}$ is connected for every $x\in Q$.
\end{lemma}

\begin{proof}
Let $Q'=x^{\lmlt(Q)}$. The rack $Q/\alpha$ is connected and so according to Corollary \ref{orbits are onto} the canonical map $y\mapsto [y]_{\alpha}$ restricted to $Q'$ is onto. In particular, $Q'\longrightarrow Q/\alpha\longrightarrow Q/\lambda_Q$ in onto since $\alpha\leq \lambda_Q$. Therefore $\setof{L_x}{x\in Q'}=\setof{L_x}{x\in Q}$. Therefore $Q'=x^{\lmlt(Q')}$ and so $Q'$ is connected.
\end{proof}

Let $Q$ be a connected quandle, $G=\dis(Q)$, $H=\dis(Q)_x$ and $f=\widehat{L_x}$. According to Proposition \ref{coset reps} we have that $Q\cong \mathcal{Q}(G,H,f)$. If $N\leq \dis(Q)$ is a normal subgroup of $\lmlt(Q)$ we can define the principal quandle $Q_N=\mathcal{Q}(G/N,f_N)$ where $f_N$ is the  automorphism induced by $f$ on the factor $G/N$. 

\begin{lemma}\label{Q_N}
Let $Q$ be a connected quandle and $N\leq \dis(Q)$ be a normal subgroup of $\lmlt(Q)$. The maps in the following diagram are surjective quandle homomorphisms and the diagram is commutative:
	\begin{center}
$\xymatrixrowsep{0.25in}
		\xymatrixcolsep{0.25in}
		\xymatrix{ 
			 Q_1  \ar@{->}[d]\ar@{->}[r]& Q_N\ar@{->}[d]&\\
			 Q  \ar@{->}[r] & Q/\mathcal{O}_N&
		},\quad \xymatrixrowsep{0.25in}
		\xymatrixcolsep{0.25in}
		\xymatrix{ 
			 g  \ar@{->}[d]\ar@{->}[r]& gN\ar@{->}[d]&\\
			 g\dis(Q)_x  \ar@{->}[r] & [g(x)]_{\mathcal{O}_N}&
		}.$
	\end{center}
 Moreover, $Q_N$ is a connected cover of $Q/\mathcal{O}_N$.
\end{lemma}

\begin{proof}
Let $G=\dis(Q)$. The subgroup $N$ is a $\widehat{L_x}$-invariant subgroup. Therefore we can define the quandle $Q_N$ and the map $$Q_1\longrightarrow Q_N,\quad g\mapsto gN,$$
is a well defined onto quandle morphism.

Let $\alpha=\mathcal{O}_N$. Assume that $g^{-1} h\in N$. Then $h=ng$ for some $n\in N$ and so we have that $h(x)=ng(x)\, \alpha \, g(x)$ for every $x\in Q$. Thus the map 
$$ Q_1\longrightarrow Q_N\longrightarrow Q/\alpha,\quad g\mapsto gN\mapsto [g(x)]_\alpha$$
is well defined. The map $gN\mapsto [g(x)]_\alpha=\pi_\alpha(g)([x]_\alpha)$ (see \eqref{pi(h)}) is onto since $Q/\alpha$ is connected and $\pi_\alpha$ is onto. Moreover $$g N *h N=g L_x g^{-1} h L_x^{-1}N\mapsto [gL_x g^{-1} h L_x^{-1}(x)]_\alpha=[L_{g(x)}h(x)]_\alpha=[g(x)]_\alpha*[h(x)]_\alpha,$$
so the map is a quandle homomorphism.

According to \cite[Proposition 3.2]{covering_paper} the quandle $Q_1$ is connected, and so it is $Q_N$.

Assume that $[g(x)]_\alpha=[h(x)]_\alpha$. Then $g^{-1}h\in \dis(Q)_{[x]_\alpha}$. Since $N$ is transitive on each block of $\alpha$ then $\dis(Q)_{[x]_\alpha}=N \dis(Q)_x$ and $\dis(Q)_x\subseteq Fix(\widehat{L_x})$. Hence, $h=g n s$ for some $n\in N$ and $s\in Fix(\widehat{L_x})$. Therefore 
\begin{align*}
hN*tN&=gnsL_x (s^{-1}n^{-1}g^{-1}t)L_x^{-1} N\\
&=g L_x g^{-1} t L_x^{-1} \underbrace{L_x (g^{-1} t)^{-1} L_x^{-1} nL_x (n^{-1}g^{-1}t)L_x^{-1} }_{\in N}N\\
&=g L_x g^{-1} t L_x^{-1}N=gN*tN.    
\end{align*}
for every $t\in G$.
Then the kernel of the map is contained in the Cayley kernel.
\end{proof}

We can apply Lemma \ref{Q_N} to the relative displacement groups.

\begin{corollary}\label{cor 1}
    Let $Q$ be a finite quandle and $G=\dis(Q)$. If $\alpha=\mathcal{O}_{\dis_\alpha}$ then $Q_{\dis_\alpha}$ is a cover of $Q/\alpha$. 
    \end{corollary}
 

\section{Quandle cohomology}    \label{cohomology}
\subsection{Constant quandle cocycles}

Let $Q$ be a rack, $S$ be a set and $\theta:Q\times Q\longrightarrow Sym_S$. We can define an operation on $E=Q\times S$ as 
\begin{equation}\label{Q x S}
    (x,a)*(y,b)=(xy,\theta_{x,y}(b))
\end{equation}
for every $x,y\in Q$ and $a,b\in S$. It is easy to check that $E$ is a rack if and only if $Q$ is a rack and
\begin{align}\label{CC}
\theta_{xy,xz} \theta_{x,z}=\theta_{x,yz}\theta_{y,z}\tag{CC}
\end{align}
holds for every $x,y,z\in Q$. A map satisfying \eqref{CC} is called {\it constant cocycle}. We denote the rack $(Q\times S,\cdot)$ where $*$ is defined as in \eqref{Q x S} as $Q\times_\theta S$ and we call the map $p:Q\times_\theta S\longrightarrow Q$ defined by $(x,a)\mapsto x$ the {\it canonical projection} onto $Q$. Note that $p$ is a rack homomorphism.

 A constant cocycle $\theta$ is a {\it quandle cocycle} if 
\begin{align}\label{quandle cocycle}
    \theta_{x,x}=1\tag{QC}
\end{align}
 for every $x\in Q$. In particular, $Q\times_\theta S$ is a quandle if and only if $Q$ is a quandle and $\theta$ is a quandle cocycle. We denote the set of quandle cocycles of $Q$ with values in $\sym_S$ as $Z^2(Q,S)$. 
 
 Covers of a connected quandle $Q$ can be constructed as $Q\times_\theta S$ for suitable $S$ and $\theta$ \cite{AG,MeAndPetr}. 

\begin{example}\label{ex trivial cov}
Let $Q$ be a rack. We denote by 
$$\textbf{1}:Q\times Q\longrightarrow \sym_S,\quad (x,y)\longrightarrow 1$$
the {\it trivial cocycle} over $S$. The rack $E=Q\times_{\textbf{1}} S$ is the direct product of $Q$ and the projection quandle over $S$. Such cover of $Q$ and is said to be a {\it trivial cover} of $Q$. Note that $(x,a)^{\lmlt(E)}=x^{\lmlt(Q)}\times \{a\}$. Hence, if $E$ is connected then $|S|=1$ and so $E\cong Q$.
\end{example}

We say that $\theta,\theta'\in Z^2(Q,S)$ are {\it cohomologous} if there exists $\gamma:Q\longrightarrow \sym_S$ such that
\begin{equation}\label{cohomologous}
\theta'_{x,y}=\gamma_{xy}\theta_{x,y}\gamma_y^{-1}    
\end{equation}
for every $x,y\in Q$. The set of equivalence classes with respect to the relation of being cohomologous is denoted by $H^2(Q,S)$ and called the {\it quandle cohomology set}. It is easy to check that the cocycles $\theta,\theta'$ satisfy \eqref{cohomologous} if and only if there exists an isomorphism $\varphi:Q \times_\theta S   \longrightarrow Q\times_{\theta'} S$ such that the following diagram is commutative:
\begin{displaymath}\xymatrixcolsep{63pt}\xymatrixrowsep{30pt}\xymatrix{
    Q\times_\theta S\ar[r]^p
    \ar[d]^{\varphi} & Q \\
    Q\times_{\theta'} S \ar[ur]^p&}  \xymatrixcolsep{63pt},\qquad \xymatrixcolsep{63pt}\xymatrixrowsep{30pt}\xymatrix{
    (x,a)\ar[r]
    \ar[d] & x \\
    (x,\gamma_x(a)) \ar[ur]&}  \xymatrixcolsep{63pt}.
\end{displaymath}
In this case, $Q\times_\theta S$ and $Q\times_{\theta'} S$ are said to be {\it equivalent covers} (as they are isomorphic in a suitable category, see \cite{Eisermann}). 

Let $Q$ be a quandle, $h\in\lmlt(Q)$, $y\in Q$ and $\theta\in Z^2(Q,S)$. Following \cite{covering_paper} with a slightly different notation, we define a formal expression $\Theta_{g(y)}$ recursively, by setting:
\[ \Theta_y=1, \qquad \Theta_{x*g(y)}=\theta_{x,g(y)}\Theta_{g(y)},\qquad \Theta_{x\ldiv g(y)}=\theta_{x,x\ldiv g(y)}^{-1}\Theta_{g(y)}. \]

\begin{lemma}\label{l:Theta}\cite[Lemma 5.5]{covering_paper}
Let $Q$ be a left quasigroup and $E=Q\times_\theta A$.
Then 
\[h(v,b)) = (\pi_{\ker{p}}(h)(v),\Theta_{\pi_{\ker{p}}(h)(v))}(b))\]
for every $h\in \lmlt(E)$.
\end{lemma}

\begin{proposition}\label{trivial for connected} 
Let $Q$ be a connected quandle, $\theta\in Z^2(Q,S)$ and $E=Q\times_\theta S$. 
The following are equivalent:
\begin{itemize}
\item[(i)] $\theta\sim\textbf{1}  $.
\item[(ii)] $\ker{p}\wedge \mathcal{O}_{\lmlt(E)}=0_E$.
\item[(iii)] $p|_{x^{\lmlt(E)}}:x^{\lmlt(E)}\longrightarrow Q$ is an isomorphism for every $x\in E$. 
\end{itemize}
\end{proposition}

\begin{proof}
(ii) $\Leftrightarrow$ (iii) According to Corollary \ref{orbits are onto} the restriction of $p$ to $x^{\lmlt(E)}=[x]_{\mathcal{O}_{\lmlt(E)}}$ is onto. Thus the restriction of $p$ is an isomorphism if and only if $\ker{p}\wedge \mathcal{O}_{\lmlt(E)}=0_E$.

(i) $\Rightarrow$ (ii) Let $\alpha=\ker{p}\wedge \mathcal{O}_{\lmlt(E)}$ and let $\varphi: Q\times_\theta S\longrightarrow Q\times_{\textbf{1}} S$ be the isomorphism defined by $\varphi(x,a)=(x,\gamma_x(a))$ for every $(x,a)\in Q\times S$. Assume that $(x,a)\,\alpha \, (y,b)$ then $x=y$ and so $\varphi(x,a)=(x,\gamma_x(a))$ and $\varphi(x,b)=(x,\gamma_x(b))$ are in the same orbit in $Q\times_\textbf{1} S$, since $\varphi(h(x,a))=\pi_{\ker{\varphi}}(h)\varphi(x,a)$ for every $h\in \lmlt(E)$. According to Example \ref{examples}(i) then $\gamma_x(a)=\gamma_x(b)$ and so $a=b$.

(ii) $\Rightarrow$ (i) Assume that $\alpha=\ker{p}\wedge \mathcal{O}_{\lmlt(E)}=0_E$. If $\pi_{\ker{p}}(h)(x)=\pi_{\ker{p}}(g)(x)$, then according to Lemma \ref{l:Theta} we have $$h(x,a)=(\pi_{\ker{p}(h)}(x),\Theta_{\pi_{\ker{p}}(h)(x)})\,\alpha\,g(x,a)=(\pi_{\ker{p}(g)}(x),\Theta_{\pi_{\ker{p}}(g)(x)}(a))$$
for every $a\in S$. Therefore $\Theta_{\pi_{\ker{p}}(h)(x)}=\Theta_{\pi_{\ker{p}}(g)(x)}$. The quandle $Q$ is connected, so given $x_0\in Q$ for every $y\in Q$ there exists $h\in\lmlt(Q)$ such that $y=h(x_0)$. The map defined as $\gamma(x_0)=1$ and $\gamma(y)=\gamma(h(x_0))=\Theta_{h(x_0)}$ is well-defined. Hence $$\gamma(xy)\textbf{1}_{x,y}\gamma(y)^{-1}=\gamma(L_x h(x_0))\gamma(h(x_0))^{-1}=\theta_{x,y}\Theta_{h(x_0)}\Theta_{h(x_0)}^{-1}=\theta_{x,y}.$$ So $\theta\sim \textbf{1}$.
\end{proof}

Clearly if $\theta\sim \textbf{1}$ then $Q\times_\theta S$ is isomorphic to $Q\times_{\textbf{1}} S$, since we have an isomorphism respecting the blocks of the canonical projection onto $Q$. If $Q$ is finite and connected also the converse holds.
\begin{corollary}\label{trivial for connected finite}
    Let $Q$ be a finite connected quandle, $\theta\in Z^2(Q,S)$ and $E=Q\times_\theta S$. The following are equivalent: 
\begin{enumerate}
    \item[(i)] $\theta\sim\textbf{1}$. %
    \item[(ii)] $|Q|=|x^{\lmlt(E)}|$ for every $x\in E$. 
     \item[(iii)] $Q\times_\theta S\cong Q\times_{\textbf{1}} S$. 
\end{enumerate}
\end{corollary}

\begin{proof}
Let $(x,a)\in E$ and $E'=x^{\lmlt(E)}$. 

(i) $\Leftrightarrow$ (ii) The canonical map restricted to $E'$ is onto. Therefore $|E'|\geq |Q|$. Hence $E'\cong Q$ if and only if $|E'|=|Q|$. Therefore (i) and (ii) are equivalent according to Proposition \ref{trivial for connected}. 

(i) $\Rightarrow$ (iii) Clear.

(iii) $\Rightarrow$ (i) Assume that $Q\times_\theta S\cong Q\times_{\textbf{1}} S$. Then $|E'|=|Q|$ and the canonical map restricted to $E'$ is onto. Then necessarily $p$ restricted to $E'$ is also injective. Thus, we can conclude by Proposition \ref{trivial for connected}.
\end{proof}

For faithful quandles we can drop the finiteness assumption in Corollary \ref{trivial for connected finite}.

\begin{lemma}\label{trivial for faithful}
Let $Q$ be a faithful quandle and $\theta\in Z^2(Q,S)$. The following are equivalent:
\begin{itemize}
\item[(i)] $Q\times_\theta S\cong Q\times_\textbf{1} S$.
\item[(ii)] $\theta\sim \textbf{1}$.

\end{itemize}
\end{lemma}

\begin{proof}

(i) $\Rightarrow$ (ii) According to Corollary 3.5 of \cite{covering_paper}, $Q\times_\theta S\cong Q\times_{\textbf{1}} S$ if and only if $\theta$ is cohomologous to $\textbf{1}\circ (g\times g)$ where $g\in \aut{Q}$. Therefore $\theta\sim \textbf{1}\circ (g\times g)=\textbf{1}$.

(ii) $\Rightarrow$ (i) Clear.
\end{proof}

Note that we can define constant cocycles with values in any group $G$ and the conditions \eqref{CC} and \eqref{quandle cocycle} still make sense and also Lemma \ref{l:Theta} can be extended to cocycles with values in arbitrary groups. The set of constant quandle cocycles of $Q$ with values in a group $G$ is denoted by $Z^2(Q,G)$. We can also define $H^2(Q,G)$ as the set of equivalence classes of $Z^2(Q,G)$ under the equivalence defined in the very same way as in \eqref{cohomologous} for some $\gamma:Q\longrightarrow G$.

So, if $Q$ is a quandle and $\theta\in Z^2(Q,G)$, we can consider $\theta$ as a cocycle with values in $\sym_G$ identifying the elements of $G$ with their canonical left action, hence we can consider $\theta\in Z^2(Q,\sym_G)$. Clearly if $[\theta]_\sim=[\textbf{1}]_\sim\in H^2(Q,G)$ then $[\theta]_\sim=[\textbf{1}]_\sim\in H^2(Q,\sym_G)$. For connected quandles also the converse holds.

\begin{corollary}   \label{H^2(G)}    
Let $Q$ be a connected quandle, $G$ be a group and $\theta\in Z^2(Q,G)$. If $[\theta]_\sim=[\textbf{1}]_\sim\in H^2(Q,\sym_G)$ then $[\theta]_\sim=[\textbf{1}]_\sim\in H^2(Q,G)$.
\end{corollary}
\begin{proof}
    Assume that $[\theta]_\sim=[\textbf{1}]_\sim\in H^2(Q,\sym_G)$. As in the proof of Proposition \ref{trivial for connected}, the map $\gamma(y)=\Theta_{h(x_0)}$ provided that $h(x_0)=y$ is a well defined map from $Q$ to $G$. And $\gamma(xy)\gamma(y)^{-1}=\theta_{x,y}$. Hence $[\theta]_\sim=[\textbf{1}]_\sim\in H^2(Q,G)$. 
\end{proof}


Let $Q$ be a latin quandle, $u\in Q$ and $\theta$ be a constant cocycle with values in $G$. We say that $\theta$ is {\it $u$-normalized} if $\theta_{x,u}=1$ for every $x\in Q$. According to \cite{MeAndPetr}, given a $\theta\in Z^2(Q,G)$, setting $\gamma_x=\theta_{x/u,u}^{-1}$, then the cocycle
\begin{align}\label{u norm}
\theta'_{x,y}=\theta_{(xy)/u,u}^{-1}\theta_{x,y}\theta_{y/u,u}    
\end{align}
is $u$-normalized. So, for latin quandles we can focus on normalized cocycles with respect to the computation of the classes of constant quandle cocycles.

\begin{proposition} \cite[Proposition 3.2, Corollary 3.3]{MeAndPetr}\label{normalized cocycles}
 Let $Q$ be a latin quandle, $u\in Q$ and $\theta\in Z^2(Q,S)$. Then $\theta$ is cohomologous to a $u$-normalized constant cocycle. Moreover, if $\theta$ is $u$-normalized and $\theta\sim\textbf{1}$ then $\theta=\textbf{1}.$   
\end{proposition}

\subsection{Abelian cocycles}
Let $A$ be an abelian group. The constant cocycle with values in $A$ are called {\it abelian cocycles} and the cover $E=Q\times_\theta A$ is called {\it abelian cover}. 
%
%
%
%

Note that quandles constructed with abelian cocycles are examples of {\it central extensions} as defined in \cite{CP} (already defined in \cite{ab_ext} under the name of abelian extensions). As a corollary of \cite[Lemma 3.5]{Galois} we have the following.

\begin{lemma}\label{cong from subs}
Let $Q$ be a quandle, $A$ an abelian group, $\theta\in Z^2(Q,A)$ and $N\leq A$. The equivalence relation defined by setting
$$(x,a)\,\sim_N\,(y,b)\, \text{ if and only if }x=y \text{ and } b-a\in N$$
is a congruence of $E=Q\times_\theta A$ and $\sim_N\leq \ker{p}$.
\end{lemma}

For connected abelian covers, the congruences below the kernel of the canonical projection coincide with the congruences defined in Lemma \ref{cong from subs}.
\begin{proposition}\label{cong from subs for connected}
Let $Q$ be a quandle, $A$ an abelian group, $\theta\in Z^2(Q,A)$ and $x\in Q$. If $E=Q\times_{\theta} A$ is connected then the map
\begin{align*}
 \psi: \setof{N}{N\leq A} \longrightarrow \setof{\alpha\in Con(E)}{\alpha\leq \ker{p}}, \quad
N\mapsto    \sim_N
\end{align*}
is a bijection.
\end{proposition}

\begin{proof}
Clearly if $\sim_N=\sim_M$ then $N=M$. So the map $\psi$ is injective. On the other hand, consider $\alpha\leq \ker{p}$. Then $\dis(E)_{[(x,0)]_\alpha}\leq \dis(E)_{[(x,0)]_{\ker{p}}}$ and so $\dis(E)_{[(x,0)]_\alpha}$ maps $=[(x,0)]_{\ker{p}}=\{x\}\times A$ to itself. Let us denote by $H_\alpha$ the subgroup $\dis(E)_{[(x,0)]_\alpha}$. Note that if $h\in H_\alpha$ then $h(x,s)=(x,s+t_h)$ for some $t_h\in A$, so we can identify $H_\alpha|_{\{x\}\times A}$ with a subgroup of $A$ acting by translations denoted by $T_\alpha$. According to \cite[Proposition 1.3]{GB} we have $[(x,0)]_\alpha=(x,0)^{H_\alpha}=(x,T_\alpha)=[(x,0)]_{\sim_{T_\alpha}}$. The quandle $E$ is connected and so congruence regular \cite[Lemma 2.3]{Maltsev_paper}. Thus $\alpha=\sim_{T_\alpha}.$ So the map $\psi$ is also onto.
%
%
\end{proof}

 Proposition \ref{cong from subs for connected} does not hold for non-connected covers. Let $Q$ be a quandle, $A$ an abelian group and $E=Q\times_{\textbf{0}} A$ where $\textbf{0}_{x,y}=0$ for every $x,y\in Q$. Then every partition of $A$ provides a congruence, so congruences below $\ker{p}$ are not in one-to-one correspondence with subgroups of $A$.

\begin{lemma}\label{A is generated}
    Let $Q$ be a quandle, $A$ an abelian group and $\theta\in Z^2(Q,A)$. If $E=Q\times_\theta A$ is connected then $A=\langle \theta_{x,y},\, x,y   \in Q\rangle$.
\end{lemma}

\begin{proof}
Let $H=\langle \theta_{x,y},\, x,y   \in Q\rangle$. Note that $(x,0)^{\lmlt(Q)}\subseteq x^{\lmlt(Q)}\times H$. Since $E=(x,0)^{\lmlt(E)}$ then $H=A$.
\end{proof}

\begin{proposition}\label{size of abelian covers}
    Let $Q$ be a finite connected quandle, $A$ abelian group and $\theta\in Z^2(Q,A)$. If $E=Q\times_\theta A$ is not trivial then $Q$ has a connected cover of size $|Q|p$ for some prime $p$.
\end{proposition}

\begin{proof}
    If $E$ is not trivial we can consider $E'=(x,0)^{\lmlt(E)}$ that is connected and $|E|>|Q|$ (see Corollary \ref{trivial for connected finite}). So $E'\cong Q\times_{\theta} A'$. The group $A'$ is generated by the set $\setof{\theta_{x,y}}{x,y\in Q}$ according to Lemma \ref{A is generated} and so it is finitely generated. In particular $A\cong \Z^k\times B$ for some $k\in \mathbb{N}$ and some finite abelian group $B$. Every subgroup of $A$ provides a congruence as defined in Lemma \ref{cong from subs}. If $k\neq 0$, let $N=\mathbb{Z}^{k-1}\times p\Z\times B$. Then $E''=E'/\sim N$ is a finite connected cover of $Q$. If $k=0$ consider $E''=E'/N$ where $N$ is a subgroup of $A'$ of prime index $p$. In both cases $|E''|=|Q|p$. 
\end{proof}

Abelian cocycles have been studied in Section 7 of \cite{covering_paper} and the cases when a cover can be constructed using abelian cocycles have been understood. In particular, we proved that connected covers of principal quandles are principal and they can be realized using abelian covers.

\begin{theorem}\label{cover of principal}\cite[Theorem 7.5]{covering_paper}
    Let $G$ be a group, $Q=\mathcal{Q}(G,f)$ be a connected quandle and $E$ be a connected cover of $Q$. Then $E$ is principal and $E$ is an abelian cover of $Q$.
\end{theorem}

\section{Simply connected quandles}\label{simply}
\subsection{Characterizations of simply connected quandles}

A quandle $Q$ is {\it simply connected} if it is connected and the covers of $Q$ are equivalent to $Q\times_{\textbf{1}} S$  for some set $S$ (in the sense of \cite{Eisermann}). Equivalently $H^2(Q,S)={\textbf{1}}$ for every set $S$ (see Proposition 2.16 in \cite{MeAndPetr}. A characterization of simply connected quandles have been obtained in \cite{Eisermann} and in \cite{covering_paper}. In particular, the first paper make use of the \emph{adjoint group} of a quandle defined as the group with presentation
$$\Adj(Q)=\langle e_x:\,x\in Q  \ |\ e_x e_y e_x^{-1}=e_{x*y}:\, x,y\in Q\rangle.$$ 
The same group is also known as the {\it enveloping group} of $Q$ and {\it structure group} of $Q$ \cite{ESS, QuadraticNichols}. The assignment $Q\mapsto \Adj(Q)$ defines the adjoint functor of the functor $\Conj(-)$.

According to \cite{Eisermann}, there exists a group homomorphism $\overline{L_Q}: \Adj(Q)\to \lmlt(Q)$ such that the following diagram is commutative
\begin{equation*}
\xymatrixcolsep{63pt}\xymatrixrowsep{30pt}\xymatrix{ Q\ar[dr]^{L_Q}
\ar[r]^{\iota} & \Adj(Q)\ar[d]^{\overline{L_Q}} \\ & \lmlt(Q) },\qquad \xymatrixcolsep{63pt}\xymatrixrowsep{30pt}\xymatrix{ x\ar[dr]
\ar[r] & e_x\ar[d] \\ & L_x }.  \label{Diag:Adjoint}
\end{equation*}

    \begin{theorem}\cite[Theorem 4.3]{covering_paper}\label{general simply}
        Let $Q$ be a quandle and $$\Adj^0(Q)=\setof{e_{x_{i_1}}^{k_1}\cdots e_{x_{i_r}}^{k_r}}{\sum_{j=1}^r k_j=0,\, r\in \mathbb{N}, x_{i_l}\in Q }\leq \Adj(Q).$$ The following are equivalent:
        \begin{itemize}
            \item[(i)] $Q$ is simply connected.
            \item[(ii)] $Q$ is principal and $\overline{L_Q}:\Adj^0(Q)\longrightarrow \dis(Q)$ is an isomorphism.
        \end{itemize}
    \end{theorem}

In particular simply connected quandles are principal by Theorem \ref{general simply}, so in this section we are dealing with quandles of the form $\mathcal{Q}(G,f)$. 

Simply connected quandles can be also characterized by using quandle cocycles with values in arbitrary groups as in \cite[Theorem 1.2]{MeAndPetr}. Using Corollary \ref{H^2(G)} we can extend the very same result to connected quandles (answering affirmatively to an open question given in \cite[Remark 9.7]{MeAndPetr}).

\begin{proposition}\label{simply then all groups}
    Let $Q$ be a connected quandle. The following are equivalent:
    \begin{enumerate}
        \item[(i)] $H^2(Q,\sym_S)=1$ for every set $S$.
        \item[(ii)] $H^2(Q,G)=1$ for every group $G$. 
 \item[(iii)] $Q$ is simply connected. 
    \end{enumerate}
\end{proposition}

We are going to offer another characterization of finite simply connected quandles.

\begin{lemma}\label{connected cover for simply 1}
    Let $Q$ be a quandle. If $Q$ is simply connected then every connected cover of $Q$ is isomorphic to $Q$. 
\end{lemma}

\begin{proof}
Let $E$ be a connected cover of $Q$. Then $E=x^{\lmlt(E)}$ for every $x\in E$ and so $E\cong Q$ according to Proposition \ref{trivial for connected}.
\end{proof}

\begin{theorem}\label{H2 abelian}
Let $Q$ be a finite quandle. The following are equivalent:
\begin{itemize}
\item[(i)] $Q$ is simply connected.
\item[(ii)] $Q$ is principal and $H^2(Q,\Z_p)=\{\textbf{1}\}$ for every prime $p$.
        \item[(iii)] Every connected cover of $Q$ is isomorphic to $Q$. 
\end{itemize} 


\end{theorem}

\begin{proof} 

(i) $\Rightarrow$ (ii) It follows by Theorem \ref{general simply} and Proposition \ref{simply then all groups}.

(ii) $\Rightarrow$ (i) Let $E$ be a connected cover of $Q$. According to Theorem \ref{cover of principal} we have that $E\cong Q\times_\theta A$ for some abelian group $A$. Hence there exists a connected cover $E'=Q\times_{\theta'} \Z_p$ for some prime $p$ (see Proposition \ref{size of abelian covers}). But $H^2(Q,\Z_p)=\{\textbf{1}\}$ and so $E'$ is trivial, contradiction.

(i) $\Leftrightarrow$ (iii) The necessary condition is Lemma \ref{connected cover for simply 1}. Let $E=Q\times_\theta S$ be a cover of $Q$ and $E'=x^{\lmlt(E)}$. According to Lemma \ref{orbits are connected}, $E'$ is a connected cover of $Q$. Then $|E'|=|Q|$ and so we can conclude by Corollary \ref{trivial for connected finite}.
\end{proof}

If a quandle has a simply connected factor, then we can identify some properties of the underlying congruence and related subgroups. 
\begin{proposition}\label{cor 2}
    Let $Q$ be a finite connected quandle and $\alpha\in Con(Q)$. If $Q/\alpha$ is simply connected then $\alpha=\mathcal{O}_{\dis_\alpha}$ and $\dis_\alpha=\dis^\alpha$.
\end{proposition}
\begin{proof}
Let $\beta=\mathcal{O}_{\dis^\alpha}  \leq \alpha$ and $G=\dis(Q)$. According to \cite[Theorem 1.10]{semimedial} the pair of operators $\alpha\mapsto \dis^\alpha$, $N\mapsto \mathcal{O}_N$ provides a Galois connection between the congruence lattice of $Q$ and the normal subgroups of $\lmlt(Q)$. So we have that $\dis^\alpha=\dis^\beta$. Thus, $\dis_\alpha\leq \dis^\alpha=\dis^\beta$ and accordingly $\pi_{\beta}(\dis_\alpha)=\dis_{\alpha/\beta}=1$ \cite[Proposition 3.2 (1)]{CP}. Thus, $\alpha/\beta\leq \lambda_{Q/\beta}$ and $(Q/\beta)/(\alpha/\beta)\cong Q/\alpha$. Therefore $Q/\beta$ is a connected cover of $Q/\alpha$. The quandle $Q/\alpha$ is simply connected, and so $|Q/\alpha|=|Q/\beta|$ by Theorem \ref{H2 abelian} and so $\alpha=\beta$.

 
According to Proposition \ref{coset reps quotient} we have $Q/\alpha\cong \mathcal{Q}(G/\dis^\alpha,f_{\dis^\alpha})$. The quandle $E=\mathcal{Q}(G/\dis_\alpha,f_{\dis_\alpha})$ is a connected cover of $Q/\alpha$ according to Corollary \ref{cor 1}. Therefore $|Q/\alpha|=[G:\dis^\alpha]=|E|=[G:\dis_\alpha]$, and so necessarily $\dis_\alpha=\dis^\alpha$.
    \end{proof}

Let us turn out attention to simply connected latin quandles. Simply connected quandles are not necessarily latin, e.g. {\tt SmallQuandle}(8,1) in the \cite{RIG} database of connected quandles is simply connected and not faithful.

In general, the property of being simply connected is not stable under taking quotients or subalgebras. For latin quandles we have the following results.

\begin{proposition}\label{3 gen}
Let $Q$ be a latin quandle and $u\in Q$. If $Sg(u,x,y)$ is simply connected for every $x,y\in Q$ then $Q$ is simply connected.
\end{proposition}
  
\begin{proof}
Let $\theta$ be a $u$-normalized cocycle, $x,y\in Q$ and $S=Sg(u,x,y)$. Then $\theta|_{S\times S}$ is a $u$-normalized cocycle and so $\theta|_{S\times S}=\textbf{1}$ since $S$ is simply connected. Therefore $\theta_{x,y}=1$.  
\end{proof}
        
\begin{proposition}\label{simply factor}
    Let $Q$ be a finite simply connected latin quandle and $\alpha\in Con(Q)$. Then $Q/\alpha$ is simply connected.
\end{proposition}

\begin{proof}
Note that $Q/\alpha$ is latin. Let $\theta$ be a $[u]$-normalized constant quandle cocycle of $Q/\alpha$ with values in $G$. Then
$$\widetilde \theta:Q\times Q\longrightarrow Q/\alpha\times Q/\alpha\longrightarrow G,\quad (x,y)\mapsto ([x],[y])\mapsto \theta_{[x],[y]}  $$
is a $u$-normalized quandle cocycle of $Q$. Thus, according to Proposition \ref{normalized cocycles} we have $\widetilde \theta=1$ and so also $\theta=1$. Thus, $Q/\alpha$ is simply connected.
\end{proof}

\subsection{Simply connected nilpotent quandles}



Finite principal quandles over nilpotent groups admits a prime power decompositions obtained by using the analog decomposition of the underlying group.

\begin{lemma}\label{direct dec}
    Let $G$ be a finite nilpotent group and $f\in \aut{G}$. Then $\mathcal{Q}(G,f)\cong \prod_{p | |G|} \mathcal{Q}(S_p,f|_{S_p})$ where $S_p$ is the unique $p$-Sylow subgroup of $G$. 
\end{lemma}

\begin{proof}
 Let $|G|=p_1^{k_1}\cdots p_r^{k_r}$. The group $G$ is the direct product of its $p$-Sylow subgroups and the $p$-Sylow subgroups of $G$ are invariant under $f$. Let $g_i:G\longrightarrow S_{p_i}$ the canonical group homomorphism  onto the Sylow subgroup $S_{p_i}$ for every $i=1,\ldots, r$. 
 It is easy to check that the map $$\mathcal{Q}(G,f) \longrightarrow \prod_{i=1}^r \mathcal{Q}(S_{p_i},f|_{S_{p_i}}),\quad x\mapsto (g_{1}(x),\ldots  g_{r}(x))$$
    is a quandle isomorphism.
    \end{proof}

Covers of principal quandles over nilpotent groups are also principal over nilpotent groups.

\begin{proposition}\label{cover of affine}
    Let $G$ be a nilpotent group, $Q=\mathcal{Q}(G,f)$ be a connected quandle and $E$ be a connected cover of $Q$. Then $E$ is nilpotent and principal over a nilpotent group.
%
%
%
\end{proposition}

\begin{proof}
 %
Note that $Q$ is nilpotent. According to Theorem \ref{cover of principal}, if $E$ is a connected cover of $Q$ i.e. $E\cong Q\times_\theta A$ for some abelian group $A$ and some abelian cocycle $\theta$. Thus, $E$ is a central extension of $Q$, and so $\ker{p}$ is a central congruence \cite[Proposition 7.5]{CP}. Therefore $E$ is nilpotent and so also $\dis(E)$ is nilpotent. Hence, $E\cong \mathcal{Q}(\dis(E),\widehat{L_x})$ according to Proposition \ref{coset reps}.
%
%
 %
\end{proof}

Let $n\in \mathbb{N}$. We define $\rad(n)=\prod_{p\in P} p$ where $P=\setof{p \text{ prime}}{p \text{ divides }n}$. For finite connected covers of principal quandles over nilpotent groups we have the following restriction on size.

\begin{proposition}\label{rad of covering}
    Let $G$ be a finite nilpotent group and $Q=\mathcal{Q}(G,f)$ and $E$ be a finite connected cover of $Q$. Then $\rad(|E|)=\rad(|Q|)$.
\end{proposition}

\begin{proof}
Let $\alpha\in Con(E)$ with $Q=E/\alpha   $ and $\alpha\leq \lambda_E$. According to Proposition \ref{cover of affine}, $E$ is principal over a nilpotent group. Therefore, $E\cong \mathcal{Q}(G',f')$ for some nilpotent group $G'$ such that $G=G'/H$ for some subgroup $H\leq G'$ (see Proposition \ref{coset reps quotient}). The block of $1$ with respect to $\alpha$ coincide with $H$ and it is a projection subquandle. Thus, $1*x=f'(x)=x$ for every $x\in H$, i.e. $H\leq Fix(f')$. Since $|Q|=|G|$ divides $|E|=|G'|=|G||H|$ then $\rad(|Q|)$ divides $\rad(|E|)$. Assume that $p$ does not divide $|Q|$. Then the $p$-Sylow of $G'$ is contained in $H$. According to Lemma \ref{direct dec} we have that $E\cong \mathcal{Q}(S_p,f')\times E'$, and $\mathcal{Q}(S_p,f')$ is a projection quandle. Thus, $\mathcal{Q}(S_p,f')$ is a factor of $E$, and so it is connected. According to Example \ref{examples}(i), $|S_p|=1$. Thus $p$ does not divide $|E|$.
\end{proof}


\begin{lemma}\label{connected cov p^n are finite}
    Let $Q=\mathcal{Q}(G,f)$ be a connected quandle where $G$ is a finite nilpotent group. Connected covers of $Q$ are finite.
\end{lemma}
\begin{proof}
    Let $E$ be a connected cover of $Q$. Then $E$ is an abelian cover of $Q$, i.e. $E\cong Q\times_{\theta} A$ for an abelian group $A$ (see Proposition \ref{cover of affine}). The group $A$ is generated by the set $\setof{\theta_{x,y}}{x,y\in Q}$ according to Lemma \ref{A is generated} and so it is finitely generated. In particular $A\cong \Z^k\times A'$ for $k\in \mathbb{N}$ and some finite abelian group $A'$. Every subgroup of $A$ provides a congruence as defined in Lemma \ref{cong from subs}. If $k\neq 0$, let $N=\mathbb{Z}^{k-1}\times q\Z\times A'$ for some prime $q$ not dividing $rad(|G|)$. Then $E'=E/\sim N$ is a finite connected cover of $Q$ and $|E'|=|Q| q$, a contradiction to Proposition \ref{rad of covering}. Hence $k=0$ and $E$ is finite.
\end{proof}


%


The following theorem provide a criterion for a finite connected quandle of prime power order to be simply connected.

\begin{theorem}\label{simply pn}
Let $Q=\mathcal{Q}(G,f)$ be a finite connected quandle and $|G|=p^n$ where $p$ is a prime and $n\in \mathbb{N}$. The following are equivalent: 
\begin{itemize}
    \item[(i)] $Q$ is simply connected.
 \item[(ii)] $H^2(Q,\Z_p)=\{\textbf{1}\}$.
     \item[(iii)]    $Q$ has no connected cover of size $p^{n+1}$.

\end{itemize}
\end{theorem}
   \begin{proof}
 (i) $\Rightarrow$ (ii) It follows by Theorem \ref{H2 abelian}.


(ii) $\Rightarrow$ (iii) Let $E$ be a connected cover of $Q$ of size $p^{n+1}$. Then $E$ is an abelian cover of $Q$, i.e. $E\cong Q\times_\theta A$ for some abelian group $A$ and $A\cong \Z_p$. If $H^2(Q,\Z_p)=\{\textbf{1}\}$ then $E\cong Q\times_\textbf{1} \Z_p$, contradiction. 

(iii) $\Rightarrow$ (i) Let $E$ be a cover of $Q$. According to Theorem \ref{cover of principal}, $E$ is an abelian cover. If $E$ is not trivial, according to Proposition \ref{size of abelian covers}, $Q$ has a connected cover $E'$ of size $|Q| q$ for some prime $q$. By Proposition \ref{rad of covering} we have that $rad(E)=p$, and so $q=p$. Thus $\theta\sim \textbf{1}$ and so $E'\cong Q\times_{\textbf{1}} \Z_p$, contradiction. Then $E$ is trivial and so $Q$ is simply connected according to Corollary \ref{trivial for connected finite}.
%
%
%
\end{proof}


The property of being simply connected for finite principal connected quandles over nilpotent groups can be checked using the prime power decomposition described in Lemma \ref{direct dec}.

\begin{theorem}\label{simply nilpotent iff}
Let $G$ be a finite nilpotent group, $\setof{S_p}{p\,|\, |G|}$ be the set of $p$-Sylow subgroups of $G$ and $Q=\mathcal{Q}(G,f)$ be a connected quandle. The following are equivalent: 
 \begin{enumerate}
     \item[(i)]  $\mathcal{Q}(G,f)$ is simply connected.
     \item[(ii)] $\mathcal{Q}(S_p,f|_{S_p})$ is simply connected for every $p$ dividing $|G|$. 
 \end{enumerate}
 
\end{theorem}

 \begin{proof}
Let $G=p_1^{k_1}\cdots p_t^{k_t}$ and $Q_i=\mathcal{Q}(S_{p_i},f|_{S_{p_i}})$. Note that $Q_i$ is the orbit of $1$ under the canonical left action of $S_{p_i}$. According to Lemma \ref{direct dec} we have that $Q\cong \prod_{i=1}^t Q_i$.

(i) $\Rightarrow$ (ii) Assume that $Q_i$ is not simply connected. Let $E_i=Q_i\times_{\theta} S$ be a connected cover of $Q_i$ with $p^s=|E_i|>|Q_i|=p^t$ for some $s>t$ (see Proposition \ref{rad of covering}). Then $E=E_i\times \left(\prod_{j\neq i} Q_j\times_{\textbf{1}} S\right)$ is a cover of $Q$. Then $E'=x^{\lmlt(E)}$ maps onto both $E_i$ and $Q$ (see Corollary \ref{cor 1}). Therefore both $|E_i|$ and $|Q|$ divide $|E'|$. Hence $p^{s-t}|Q|$ divides $|E'|$ and so $|E'|>|Q|$. Hence, we can conclude by Corollary \ref{trivial for connected finite} that $Q$ is not simply connected.
%

(ii) $\Rightarrow$ (i) Assume that $Q_i$ is simply connected for every $i=1\ldots t$. Let $E$ be a connected cover of $Q$. According to Lemma \ref{connected cov p^n are finite} $E$ is finite. According to Proposition \ref{cover of affine} $E$ is principal and so $E\cong \mathcal{Q}(G',f)$ for some finite group $G'$ such that $G=G'/H$ for some subgroup $H'\leq G'$. We just need to show that $|E|=|Q|$ (see Corollary \ref{trivial for connected finite}). 

According to Proposition \ref{rad of covering} $\rad(|E|)=\rad(|Q|)$. The quandle $E$ is the direct product of principal quandles over the $p_i$-Sylow subgroups of $G'$ (using again Lemma \ref{direct dec}). If $h\in S_{p_i}'$ then $[h(x)]=\pi_\alpha(h)([x])$, so the canonical map $x\mapsto [x]$ from $E$ to $Q$ maps the orbits of $S_{p_i}'$ onto the orbits of $S_{p_i}$. Hence, we have that $E_i=\mathcal{Q}(S_{p_i}',f'|_{S_{p_i}'})$ is mapped onto $Q_i$. So $E_i$ is a connected cover of $Q_i$. Since $Q_i$ is simply connected then $|E_i|=|Q_i|$ for every $i=1, \ldots, t$ by virtue of Theorem \ref{H2 abelian}. Therefore $|E|=|Q|$. 
%
%
 %
   %
%
\end{proof}

As a corollary of the results of this section we obtain a restriction on size of $2$-step solvable latin quandles.

\begin{corollary}
    Let $Q$ be a finite $2$-step solvable latin quandle. Then $\rad(|Q|)=\rad(|\dis(Q)|)$.
\end{corollary}

\begin{proof}
Let $G=\dis(Q)$ and $\alpha=\gamma_Q$. Clearly $\rad(|Q|)$ divides $\rad(|G|)$ since $|G|=|Q||G_x|$. Let $p$ be a prime dividing $|G|=|G/\dis^\alpha||\dis^\alpha|$. The quandle $Q$ is latin, so the blocks of $\alpha$ are connected and then $\alpha=\mathcal{O}_{\dis_\alpha}$ \cite[Lemma 2.3]{Nilpotent}. The congruence $\alpha$ is abelian and so $\dis_\alpha\leq \dis^\alpha$ is regular on $[x]$. Thus $|\dis_{\alpha}|$ divides a power of $|[x]|$ according to \cite[Lemma 1.4(ii)]{LSS}. If $p$ divides $|\dis^\alpha|$ then it divides $|[x]|$ and so $p$ divides $|Q|=|Q/\alpha||[x]|$. The quandle $Q_{\dis_\alpha}$ is a connected cover of $Q/\alpha$ (see Lemma \ref{Q_N}) and $|Q_{\dis_\alpha}|=|G/\dis^\alpha|$. According to Proposition \ref{rad of covering} $ \rad(|Q_{\dis_\alpha}|)=\rad(|Q/\alpha|)$. If $p$ divides $|G/\dis_\alpha|$ then $p$ divides $|Q/\alpha|$ and so it divides $|Q|$. Therefore $\rad(|G|)$ divides $\rad(|Q|)$.
\end{proof}
\subsection{Simply connected latin quandles of size \texorpdfstring{$p^2$}{p^2} and \texorpdfstring{$p^3$}{p^3}}

We are going to exploit Theorem \ref{simply pn} in order to classify simply connected quandles of size $p^2$ and $p^3$. The following strategy can be used to obtain the classification of simply connected quandles of size $p^n$ provided the classification of the groups of size $p^n$ and $p^{n+1}$ and their automorphisms. We display the algorithm in full generality assuming that such information is available. The idea works as follows:


\begin{itemize}
    \item[(i)] Let $Q=\mathcal{Q}(G,f)$ be a connected principal quandle of size $p^{n}$. 

    \item[(ii)] Let $G'$ be a group of size $p^{n+1}$, $f\in \aut{G'}$ and $Q'=\mathcal{Q}(G',f)$. Using Proposition \ref{coset reps quotient} and \cite[Corollary 5.31]{GB} we have that $Q$ is isomorphic to a factor of $Q'$ if and only if $G\cong G'/H$ for some $H\leq G'$ of size $p$ and such that $f'_H$ is conjugate to $f$. 
    The quandle $Q'$ is a covering of $Q$ if and only if $H\leq Z(G')\cap Fix(f)$. Indeed the kernel of the map $$p:Q'\longrightarrow Q, \quad x\mapsto xH$$ is contained in $\lambda_{Q'}$ if and only if the blocks are projection, namely $H=[1]_{\ker{p}}\leq Fix(L_1)=Fix(f)$. Moreover, according to \cite[Corollary 3.3]{covering_paper} $\dis^{\lambda_{Q'}}\leq Z(\lmlt(Q))\cap \dis(Q)\leq Z(\dis(Q))$.
    
    According to \cite[Lemma 2.2]{GB} we have also that $Fix(f)\leq \gamma_1(G)$. Hence $H\leq Z(G)\cap \gamma_1(G)\cap Fix(f)$. Finally, $Q'$ is connected if and only if $Fix(f_{\Phi(G')})=1$ (see Corollary 2.8 of \cite{GB}).

    \item[(iii)] If a pair $(G',f)$ satisfying the properties listed in the previous steps exists then $Q$ is not simply connected, otherwise $Q$ is simply connected (see Theorem \ref{simply pn}).


\end{itemize}

Note that the property of being simply connected for latin quandles can be also checked by using the necessary condition stated in Proposition \ref{simply factor}. Indeed, given $Q$ a latin quandle of size $p^n$ and the list of simply connected latin quandles of size $p^{k}$ for $k\leq n-1$, if $Q$ has a factor that is not in the list (up to isomorphism), then $Q$ is not simply connected. 

   The groups of size $p^{k}$ for $k=3,4$ and their automorphisms are classified in \cite{tedesco} and we collected the data we need in Section \ref{appendix}. Moreover connected quandles of size $p^2$ and $p^3$ are classified in \cite{Hou} and \cite{GB}. So we can apply the algorithm above for the case $p^2$ and $p^3$.

According to one of the main resuls of \cite{MeAndPetr}, affine quandles over cyclic groups are simply connected, so from now on we can omit that case.

\begin{proposition}\cite[Theorem 1.1]{MeAndPetr}\label{cyclic are simply}
   Let $Q = \aff(\Z_m,f)$ be a connected quandle. Then $Q$ is simply connected.
\end{proposition}

Let us start with the case $p^2$. Non-faithful quandles of order $p^3$ are classified. If $Q$ is one of such quandles, then $Q/\lambda_Q$ have size $p^2$ (quandles of size $p$ are simply connected, see Proposition \ref{cyclic are simply}). 

Let us list the connected quandles of size $p^2$ such that the displacement group is not cyclic as provided in \cite{Hou} in Table \ref{affine p^2} and the principal non-faithful quandles of size $p^3$ in Table \ref{Tab2} (as given in Table 1 of \cite{GB}).

\begin{table}[!htb]
    \centering
    \caption{Connected affine quandles over $\Z_p^2$ where $p$ is prime are of the form $\aff(\Z_p^2,f)$ where $f$ is one of the following:}
    \begin{tabular}{|l|l|}
        \hline
        $f$ & Parameters \\
        \hline
         $D(b,c)=\begin{bmatrix} b & 0 \\0 & c \end{bmatrix}$ & $1 < b \leq c < p$. \\
        \hline
         $G(b)=\begin{bmatrix} b & 1 \\ 0 & b \end{bmatrix}$ & $1< b< p$. \\
        \hline
        $H(q)=\begin{bmatrix} 0 & 1 \\ -b_0 & -b_1 \end{bmatrix}$ & $q=x^2 + b_1x + b_0$  is an irreducible polynomial over $\Z_p$. \\
          \hline
    \end{tabular}   \label{affine p^2}
    \end{table}

\begin{table}[!htb]
\caption{The principal non faithful of size $p^3$ with $p> 3$ are of the form $\mathcal{Q}(\Z_p^2\rtimes \Z_p,f)$ where $f$ is one of the following:}
\begin{tabular}{|l|l|l|c|c|}
\hline
$f$  & Parameters    \\
 \hline
 $\widetilde{D}(b,b^{-1}):\, a_1\mapsto a_1^{b}$, \quad  $a_2\mapsto a_2^{b^{-1}}$ & $0<b\leq  b^{-1} <p-1$. \\
 \hline
 $\widetilde{G}(-1):\, a_1\mapsto a_1^{-1}$, \,\,  $a_2\mapsto a_1 a_2^{-1}$ &   \\
 \hline
  $\widetilde{H}(q):\, a_1\mapsto a_2$, \quad  $a_2\mapsto  a_1^{-1} a_2^{-b}$ & $q=x^2+bx+1$ is an irreducible polynomial over $\Z_p$.   \\
 \hline
\end{tabular}\label{Tab2}

\end{table}

\begin{theorem}\label{simply p2}
Let $Q=\aff(\Z_p^2,f)$ be a connected quandle. Then $Q$ is simply connected if and only if $f$ is one of the following: 
\begin{itemize}
 \item[(i)]  $D(a,b)$ where $ab\neq 1 \pmod{p}$.
 \item[(ii)] $G(c)$, where $c\neq - 1 \pmod{p}$.
 \item[(iii)] $H(q)$, where $q=x^2+b_1x+b_0$ with $b_0\neq 1$. 


\end{itemize}

\end{theorem}

\begin{proof}

Let $G=\Z_p^2\times \Z_p$, $A=\Z_p^2$ and $$\psi:G\longrightarrow A,\quad x\mapsto xZ(G).$$
It is easy to see that: 

\begin{itemize}
    \item $\psi$ is a covering map from $\mathcal{Q}(G,\widetilde{D}(b,b^{-1}))$ to $\aff(\Z_p^2,D(b,b^{-1}))$.
    
    \item $\psi$ is a covering map from $\mathcal{Q}(G,\widetilde{G}(-1))$ to $\aff(\Z_p^2,G(-1))$.

    \item $\psi$ is a covering map from $\mathcal{Q}(G,\widetilde{H}(q))$ to $\aff(\Z_p^2,H(q))$ if $q=x^2+bx+1$ is an irreducible polynomial over $\Z_p$.
\end{itemize}

The non-faithful quandles of order $p^3$ are principal of the form $Q=\mathcal{Q}(\Z_p^2\rtimes \Z_p,f)$ as in Table \ref{Tab2} where $Z(G)=\gamma_1(G)=Fix(f)$ and $G/Z(G)=\Z_p^2$. So $Q/\lambda_Q\cong \aff(\Z_p^{2},f_{Z(G)})$ and $f_{Z(G)}$ is conjugate to one of the automorphisms in Table \ref{affine p^2}. Since we have that $f(z)=z^a$ where $a=\det(f_{Z(G)}\pmod{p}$ for every $f\in \aut{G}$ we have that $a=1$. The determinant is invariant under conjugation, so:
\begin{itemize}
    \item if $f_{Z(G)}$ is conjugate to $D(a,b)$ then $ab=1\pmod{p}$.
    
    \item if $f_{Z(G)}$ is conjugate to $G(c)$ then $c=-1\pmod{p}$ ($c^2=1$ and $c\neq 1$).

    \item if $f_{Z(G)}$ is conjugate to $H(q)$ then $b_0=1$.
\end{itemize}
So the statement follows.
\end{proof}

Note that Proposition \ref{simply factor} does not hold for simply connected quandles in general. Indeed the quandle {\tt SmallQuandle(27,1)} in the \cite{RIG} database of connected quandles is simply connected and not latin, and it has a factor isomorphic to $\aff(\Z_p^2,-1)$ that is not simply connected according to Theorem \ref{simply p2}.

  

We now consider the case $p^3$. In the following we denote by $G_i$ the group of size $p^4$ for $i=1,\ldots ,15$ following \cite{tedesco} as described in Section \ref{appendix}. Let us first list the connected affine quandles of size $p^3$ such that the displacement group is not cyclic in Tables \ref{p^3 not elementary} and \ref{p^3 elementary} (such quandles are given in \cite{Hou}).

    \begin{table}[!ht] 
    \centering
 \caption{Connected affine quandles over $\Z_{p^2}\times \Z_p$ where $p$ is prime are of the form $\aff(\Z_{p^2}\times \Z_p,f)$ where $f$ is one of the following:}  
    \begin{tabular}{|l|l|}
        \hline
               $f$ & Parameters \\
\hline
         $D(b,c)=\begin{bmatrix} b & 0\\ 0& c \end{bmatrix}$ & $0<b<p^2-1, b\neq  0,1 \pmod{p}$ and $1<c<p$. \\
        \hline
         $G(b)=\begin{bmatrix} b & 0 \\ 1 & b \end{bmatrix}$ & $1 < b < p$. \\
        \hline
  $H(b,c)=\begin{bmatrix} b & p \\ c & b \end{bmatrix}$ & $ 0 < b < p$ and $0  \leq c< p-1 $. \\
        \hline
\end{tabular} \label{p^3 not elementary}
\end{table}

        \begin{table}[!ht]  
    \centering
     \caption{Connected affine quandles over $\Z_p^3$ where $p$ is prime are of the form $\aff(\Z_p^3,f)$ where $f$ is one of the following:}
    \begin{tabular}{|l|l|}
    \hline

        $f$ & Parameters \\
\hline

             $D(b,c,d)=\begin{bmatrix} b & 0&0\\ 0& c &0 \\ 0 &0 & d \end{bmatrix}$ &$ 1 < b \leq c \leq d < p$. \\
        \hline
         $G(b,c)=\begin{bmatrix} b & 1 & 0\\ 0& b & 0 \\ 0& 0& c \end{bmatrix}$ &$ 1<b,c<p$. \\
        \hline 
        $F(b)=\begin{bmatrix} b & 1 & 0 \\ 0& b & 1 \\ 0& 0& b \end{bmatrix}$&$ 1<b<p$. \\
        \hline
        $H(q)=\begin{bmatrix} 0 & 1 & 0 \\ 0 & 0 & 1 \\ -b_0 & -b_1 & -b_2 \end{bmatrix}$ &$ q=x^3 + b_2x^2 + b_1x + b_0$ is an irreducible polynomial over $\Z_p$. \\
        \hline
        $H(q,c)=\begin{bmatrix} 0 & 1 & 0 \\ -b_0 & -b_2 & 0 \\ 0 & 0 & c \end{bmatrix}$ &$q= x^2 + b_1x + b_0$ is an irreducible polynomial over $\Z_p$ and $1<c<p$. \\
        \hline
    \end{tabular}\label{p^3 elementary}
\end{table}

Now we compute the affine simply connected quandles of size $p^3$.
\begin{lemma}   \label{p^3 p^2 x p}
    Let $Q=\aff(\Z_{p^2}\times \Z_p,f)$ be a connected quandle. Then $Q$ is simply connected if and only if $f$ is one of the following:
\begin{enumerate}

    \item[(i)] $D(b,c)$
    for $bc\neq 1 \pmod{p}$. 
     \item[(ii)] $G(b)$ 
        for $b\neq -1 \pmod{p}$.
 \item[(iii)] $H(q)$ 
    for $b\neq -1\pmod{p}$. 
 \end{enumerate}

\end{lemma}

\begin{proof}
Let $A=\Z_{p^2}\times \Z_p$ and $Q=\aff(A,f)$. The quandle $Q$ is not simply connected if and only if $Q$ has a connected principal cover of size $p^4$. Let us call such cover by $E$. In particular $E=\mathcal{Q}(G,g)$ with $G/Fix(g)\cong A $, $Fix(g)\leq Z(G)\cap \gamma_1(G)$ and $g_{Fix(g)}$ is conjugate to $f$. According to Lemma \ref{which groups} and the description of the groups of size $p^4$ given in Section \ref{appendix}, $G=G_3$ and $Fix(g)=\langle a_3\rangle$. Moreover, using Lemma \ref{what auto?} we have that 
\begin{align*}
    g_{Fix(g)}=\begin{cases}a_1\mapsto a_1^{u_1} a_2^{u_2} a_4^{u_4}=a_1^{u_1+pu_4} a_2^{u_2} ,\\
    a_2\mapsto a_2^{v_2}a_4^{v_4},=a_2^{v_2}a_1^{pv_4}\\
    a_4\mapsto a_4^{u_1},
    \end{cases}
\end{align*}
for $u_1 v_2=1\pmod{p}$. Note that $g_{\Phi(G)}=f_{\Phi(A)}$ is a uppetriangular matrix with $u_1$ and $v_2$ on the diagonal. We need to compare $g_{Fix(g)}$ with the list in Table \ref{p^3 not elementary} (up to conjugation).
\begin{itemize}
\item Assume that $f$ is conjugate to $D(b,c)$. Then we also have $\aff(A/\Phi(A),f_{\Phi(A)})\cong \aff(\Z_p^2,f_{\Phi(A)})\cong \aff(\Z_p^2,D(b,c)_{\Phi(A)})$. According to \cite[Corollary 5.31]{GB}, $f_{\Phi(A)}$ and $D(b,c)_{\Phi(A)}$ are conjugate, so $$bc=\det(D(b,c)_{\Phi(A)})=\det(f_{\Phi(A)})=u_1 v_2=1\pmod{p}.$$  On the other hand, if $f=D(b,c)$, with $bc=1\pmod{p}$ we can take $u_1=v_2=b$, $u_2=u_4=v_4=0$ to obtain a connected cover of $\aff(A,D(b,c))$.  

\item Assume that $f$ is conjugate to $G(b)$. Then, using the same argument we used in the previous case, necessarily we have that $b^2=v_1 u_2=1\pmod{p}$. On the other hand, if $f=G(b)$ with $b^2=1\pmod{p}$ we can take $u_1=v_2=b$, $u_2=1$, $u_4=v_4=0$ to obtain a connected cover of $\aff(A,G(b))$.

\item Assume that $f$ is conjugate to $H(b,c)$. Then again we have $b^2=v_1 u_2=1\pmod{p}$. On the other hand, if $f=H(b,c)$ with $b^2=1\pmod{p}$ we can take $u_1=v_2=b$, $u_2=c$, $u_4=0$ and $v_4=1$ to obtain a connected cover of $\aff(A,H(b,c))$.  

\end{itemize}
So, the statement follows.
\end{proof}

\begin{lemma}
    Let $Q=\aff(\Z_p^3,f)$ be a latin quandle. Then $Q$ is simply connected if and only if $f$ is one of the following:
\begin{enumerate}
    \item[(i)] $f=D(b_1,b_2,b_3)$ 
    with $b_i b_j \neq 1$ for every $i,j$ and $i\neq j$. 
         \item[(ii)] $f=G(b,c)$ 
            with $b\neq -1$. 
             \item[(iii)] $f=F(b)$
            with $b\neq -1$. 
       
            \item[(iv)] $f=H(q,c)$ 
            where $b_0\neq 1$.

                 \item[(v)] $f=H(q)$. 
 \end{enumerate}

\end{lemma}

\begin{proof}
As in the proof of Lemma \ref{p^3 p^2 x p}, the quandle $Q$ is not simply connected if and only if there exists a connected cover $E=\mathcal{Q}(G,g)$ such that $G/Fix(g)\cong \Z_{p}^3$, $Fix(g)\leq Z(G)\cap \gamma_1(G)$ and $g_{Fix(g)}$ is conjugate to $f$. Accordingly, $G=G_i$ for $i=12,13$ (see Section \ref{appendix}).     

So, if $i=12$ then $Fix(g)=\langle a_4\rangle$. Moreover, using Lemma \ref{what auto?} we have that 
\begin{align*} 
   g_{Fix(g)}=\begin{cases} a_1 \mapsto a_1^{u_1} a_2^{u_2} a_3^{u_3} , \\
    a_2 \mapsto a_1^{v_1} a_2^{v_2} a_3^{v_3} , \\
    a_3 \mapsto a_3^{w_3}  , 
    \end{cases}
\end{align*}
where $v_2 u_1 - v_1 u_2=1\pmod{p}$ and $u_1+v_2\neq 2$.

Similarly, if $i=13$ then
\begin{align*} 
g_{Fix(g)}=\begin{cases}     
      a_1 \mapsto a_1^{u_1} a_2^{u_2} a_3^{u_3} , \\
      a_2 \mapsto a_2^{v_2} a_3^{v_3} , \\
      a_3 \mapsto a_3^{w_3} , 
    \end{cases}
\end{align*}
where $u_1 v_2=1\pmod{p}$.    Hence comparing such automorphism with the list in Table \ref{p^3 elementary} (up to conjugation) the statement follows as in Lemma \ref{p^3 p^2 x p}.
\end{proof}

Let us consider the non-affine quandles of size $p^3$. The list of such quandles is given in Table \ref{Tab2} and \ref{Tab3}.

\begin{table}[!htb]
\caption{The principal non-affine faithful connected quandles of size $p^3$ with $p> 3$ are of the form $\mathcal{Q}(\Z_p^2\rtimes \Z_p,f)$ where $f$ is one of the following:}
\begin{tabular}{|l|l|l|c|c|}
\hline
 $f$ & Parameters  \\
 \hline $\widetilde{D}(b,c):\, a_1\mapsto a_1^b$, \quad  $a_2\mapsto a_2^c$ & $1< b \leq c< p-1$, $bc \neq 1 \pmod{p}$. \\
 \hline
 $\widetilde{G}(c):\, a_1\mapsto a_1^c$, \quad  $a_2\mapsto a_1 a_2^c$ & $1<c<p-2$.  \\
 \hline
  $\widetilde{H}(q):\, a_1\mapsto a_2$, \quad  $a_2\mapsto  a_1^{-b_0} a_2^{-b_1}$ & $q=x^2+b_1 x+b_0$ is an irreducible polynomial over $\Z_p$ and  $b_0\neq 1$.  \\
 \hline
\end{tabular}\label{Tab3}

\end{table}

\begin{lemma}\label{latin simply non affine}
    Let $Q=\mathcal{Q}(\Z_p^2\rtimes \Z_p,f)$ be a connected quandle. Then $Q$ is simply connected if and only if $f$ is one of the following:
\begin{enumerate}
    \item[(i)] $f=\widetilde D(b,c)$
   such that $b^2 c \neq 1 \mod{p}$.
 \item[(ii)] $f=\widetilde G(c)$ such that $c^3\neq 1$. 
    \item[(iii)] $f=\widetilde H(q)$. 

\end{enumerate}

\end{lemma}

\begin{proof}
  The quandle $Q$ is not simply connected if and only if $Q$ has a connected principal cover of size $p^4$. Let us call such cover by $E$. In particular $E=\mathcal{Q}(G,g)$ such that $G/Fix(g)\cong \Z_{p^2}\times \Z_p$, $Fix(g)\leq Z(G)\cap \gamma_1(G)$ and $g_{Fix(g)}$ is conjugate to $f$. According to Lemma \ref{which groups} and the description of the groups of size $p^4$ given in Section \ref{appendix}, $G=G_7$ and $Fix(g)=\langle a_4\rangle$. Using Lemma \ref{what auto?} we have that 
\begin{align}\label{check1}
    g_{Fix(g)}=\begin{cases}
    a_1\mapsto a_1^{u_1} a_2^{u_2} a_3^{u_3},\\
    a_2\mapsto a_2^{v_2}a_3^{v_3},\\
    a_3\mapsto a_3^{u_1^2 v_2},
    \end{cases}
\end{align}
for $u_1^2 v_2=1\pmod{p}$. Hence we can conclude by comparing (up to conjugation) the automorphisms as in \eqref{check1} with the automorphisms listed in Tables \ref{Tab2} and \ref{Tab3} (as in Lemma \ref{p^3 p^2 x p}).
\end{proof}



\section{Involutory quandles}

\subsection{Split quandle cocycles}

Let $Q$ be a quandle and $\theta\in Z^2(Q,G)$ a constant quandle cocycle. We say that $\theta$ is a {\it split cocycle} if $\theta_{x,y}=\rho(x)\rho(y)^{-1}$ for some $\rho:Q\longrightarrow G$. In this case we denote $\theta$ as $\theta(\rho)$. 

Split quandle cocycles arise for instance from morphisms of quandles with values in twisted conjugation quandles and core quandles.


    

\begin{lemma}\label{twisted}
    Let $G$ be a group, $f\in \aut{G}$ and $Q$ be a quandle and $\rho:Q\longrightarrow \Conj_f(G)$ be a quandle morphism. Then $\theta(\rho)\in Z^2(Q,G)$.
\end{lemma}

\begin{proof}
    Note that
    \begin{align*}
\theta_{xy,xz}\theta_{x,z}&=\rho(xy)\rho(xz)^{-1}\rho(x)\rho(z)^{-1}\\
&=\rho(x)f(\rho(y)\rho(x)^{-1})f(\rho(x))f(\rho(z)^{-1})\rho(x)^{-1}\rho(x)\rho(z)^{-1}\\
&=\rho(x)f(\rho(y))f(\rho(z)^{-1})\rho(z)^{-1}\\
\theta_{x,yz}\theta_{y,z}&=\rho(x)\rho(yz)^{-1}\rho(y)\rho(z)^{-1}\\
&=\rho(x)f(\rho(y))f(\rho(z)^{-1})\rho(y)^{-1}\rho(y)\rho(z)^{-1}\\
&=\rho(x)f(\rho(y))f(\rho(z)^{-1})\rho(z)^{-1}
    \end{align*}
    for every $x,y,z$.
\end{proof}

Note that Lemma \ref{twisted} includes both conjugation and affine quandles.

\begin{lemma}\label{cocycle by rho core}
Let $Q$ be a quandle, $G$ be a group and $\rho:Q\longrightarrow Core(G)$ be a quandle morphism. Then $\theta(\rho)\in Z^2(Q,G)$.
\end{lemma}

\begin{proof}
Let $x,y,z\in Q$, then we have 
\begin{align*}
\theta_{xy,xz}\theta_{x,z}&=\rho(xy)\rho(xz)^{-1}\rho(x)\rho(z)^{-1}=\rho(x)\rho(y)^{-1}\rho(x)\rho(x)^{-1}\rho(z)\rho(x)^{-1}\rho(x)\rho(z)^{-1}\\
&=\rho(x)\rho(y)^{-1},\\
\theta_{x,yz}\theta_{y,z}&=\rho(x)\rho(yz)^{-1}\rho(y)\rho(z)^{-1}=\rho(x)\rho(y)^{-1}\rho(z)\rho(y)^{-1}\rho(y)\rho(z)^{-1}\\
&=\rho(x)\rho(y)^{-1}.
\end{align*}
So $\theta$ is a constant cocycle.
\end{proof}

%




Let us make some examples of split cocycles.
\begin{example}\text{}
\begin{itemize}
    \item[(i)] Let $Q,$ be a quandle and $Q'=Q/\alpha$ for some $\alpha\in Con(Q)$. The map $\rho:x\mapsto L_{[x]}$ is a morphism of quandle from $Q$ to $\Conj(\lmlt(Q/\alpha))$. In this case we have
    $$(x,[a])(y,[b])=(xy,[x]([y]\backslash [b]))$$
     for every $x,y,a,b\in Q$.
     \item[(ii)] Let $G$ be a group, $f\in \aut{G}$ and $1:G\longrightarrow G$ the identity map. Then $\theta(1)$ is a split cocycle of $\core(G)$ and $\Conj_f(G)$.

\item[(iii)] Let $Q=\aff(A,f)$ be a latin quandle. Note that $x/0=(1-f)^{-1}(x)$ for every $x\in A$. Consider $\theta(1)$ and 
let $\gamma(x)=-\theta(1)_{x/0,0}$ for every $x\in A$. Then the $0$-normalized cocycle defined as in \eqref{u norm} is trivial. Indeed
\begin{align*}
\theta'_{x,y}&=\gamma_{xy} +\theta_{x,y}(1)-\gamma_y=-x-(1-f)^{-1}f(y)+x-y+(1-f)^{-1}(y)=0.
\end{align*}

\item[(iv)] Let $p$ be a prime, $Q=\aff(\Z_{p^\infty},1+p)$. The quandle $Q$ is not faithful, indeed $L_0=L_x$ provided $px=0$. Let $x\in \mathbb{Z}_{p^\infty}$ such that $px=0$. Then $L_{x} L_0^{-1}(0)=px=0$. Consider $E=Q\times_{\theta(1)} \Z_{p^\infty}$. Then 
$$L_{(x,0)}(0,0)=(px,\theta(1)_{x,0}(0))=(0,x)$$
 Hence $\mathcal{O}_{\lmlt(E)}\wedge\ker{p}\neq 0_E$, so $\theta(1)$ is not cohomologous to \textbf{1}.
 \end{itemize}

\end{example}

Let us show an explicit example of non-trivial constant split quandle cocycle for core quandles. 

\begin{lemma}\label{b(1)=1}
    Let $G$ be a group and $1:G\longrightarrow G$ be the identity map. If $\theta(1)\sim \textbf{1}$ then $\core(G)$ is superfaithful. 
\end{lemma}

\begin{proof}
Consider $E=Q\times_{\theta(1)} G$. Assume that $x\in G$ such that $x^2=1$. Then $L_x (1)=x^2=1$ and so $$(x,a)(1,1)=(x^2,x)=(1,x).$$ Then we have $(1,x)\, \ker{p}\wedge \mathcal{O}_{\lmlt(E)}\, (1,1)$ and so $x=1$ since $\ker{p}\wedge \mathcal{O}_{\lmlt(E)}=0_E$. According to \cite[Lemma 4.3]{Cores} the quandle $\core(G)$ is superfaithful. 
\end{proof}

According to \cite[Proposition 4.4]{Cores} a core quandle over a group $G$ is latin if and only if the map $s:x\mapsto x^2$ is bijective. Let us denote by $x\mapsto x^{1/2}$ the inverse map of $s$. 
\begin{proposition}\label{non trivial cocycle for cores}
    Let $Q=Core(G)$ be latin. Then $\theta(1)\sim \textbf{1}$ if and only if $G$ is abelian. 
\end{proposition}

\begin{proof}
    Note that $R_1(x)=x^2$ and so $x/1=x^{1/2}$ for every $x\in Q$. We define $\gamma_x=\theta_{x/1,1}^{-1}(1)$ as \eqref{u norm} and so we have 
\begin{align*}
    \theta'_{x,y}&=\gamma_{x  y}\theta_{x,y}(1)\gamma_y^{-1} = ((xy^{-1} x)^{1/2})^{-1} xy^{-1} y^{1/2} \\
    &= ((xy^{-1} x)^{-1})^{1/2} xy^{-1} y^{1/2} = (x^{-1}y x^{-1})^{1/2}xy^{-1}y^{1/2}.
\end{align*}        
    So $\theta_{x,y}'=(x^{-1}y x^{-1})^{1/2}(xy^{-1})y^{1/2}=1$ for every $x,y\in G$ if and only if $xy^{1/2}=y^{1/2}x$ for every $x,y\in G$. 

    Since $\theta(1)\sim \textbf{1}$ if and only if $\theta'=\textbf{1}$ (see Proposition \ref{normalized cocycles}), then $\theta(1)\sim \textbf{1}$ if and only if $G$ is abelian. 
\end{proof}

\subsection{Simply connected involutory quandles}

Recall that a quandle $Q$ is involutory if $x(xy)=y$ for every $x,y\in Q$. For instance core quandles are involutory. In this section we take a look to simply connected involutory quandles. Recall that finite involutory connected affine quandles are of the form $\aff(A,-1)=\core(A)$ for some finite abelian group $A$ of odd size.

\begin{proposition}\label{p involutory are cyclic}
    Let $G$ be a finite $p$-group. If $Q=\mathcal{Q}(G,f)$ is a latin involutory simply connected quandle then $G$ is cyclic. 
\end{proposition}

\begin{proof}
The Frattini subgroup $\Phi(G)$ is characteristic and so the left coset partitions of $G$ with respect to $\Phi$ provides a congruence of $Q$. Assume $G/\Phi(G)\cong \Z_p^n$ and $n\geq 2$, then $Q$ has a factor that is principal over $G/\Phi$ and so it is isomorphic to $Q'=\Aff(\Z_p^n,-1)$. Every subgroup of $\Z_p^n$ provides a congruence of $Q'$, then $Q$ has a factor isomorphic to $Q''=\aff(\Z_p^2,-1)$ and $Q''$ is simply connected (see Proposition \ref{simply factor}). A contradiction to Theorem \ref{simply p2}. So $G/\Phi(G)$ is cyclic and therefore also $G$ is cyclic.
\end{proof}

We can use the direct decomposition of principal quandles over nilpotent groups to characterize finite nilpotent involutory latin quandles that are simply connected. 
\begin{theorem}\label{involutory simply}
    Let $Q$ be a finite nilpotent involutory latin quandle. The following are equivalent:
    \begin{itemize}
        \item[(i)] $Q$ is simply connected.
        \item[(ii)] $Q\cong \core(\Z_{m})$ for some odd integer $m\in \mathbb{N}$. 
    \end{itemize}
\end{theorem}
\begin{proof}
    (i) $\Rightarrow$ (ii) According to Theorem \ref{general simply} the quandle $Q$ is principal and according to Proposition \ref{coset reps} we have $Q\cong \mathcal{Q}(\dis(Q),\widehat{L_x})$. The quandle $Q$ is nilpotent and so $\dis(Q)$ is nilpotent. According to Lemma \ref{direct dec} we have that the quandle $Q$ is a direct product of $\prod_{p} \mathcal{Q}(S_p,f|_{S_p})$. Hence $\mathcal{Q}(S_p,f|_{S_p})$ is simply connected (see Theorem \ref{simply nilpotent iff}) and so $S_p$ is cyclic according to Proposition \ref{p involutory are cyclic}. Then $G$ is cyclic since it is the direct product of cyclic groups with coprime size.
    
    (ii) $\Rightarrow$ (i) It follows by Proposition \ref{cyclic are simply}.
\end{proof}

We can characterize finite simply connected core quandles.
\begin{corollary}   \label{simply cores}
    Let $G$ be a finite group. Then $\core(G)$ is simply connected if and only if $G$ is cyclic of odd size.
\end{corollary}

\begin{proof}
Assume that $Q=\core(G)$ is simply connected. Then $Q$ is superfaithful by Lemma \ref{b(1)=1}. Hence $Q$ is latin and it has odd size according to \cite[Corollary 4.5]{Cores}. Moreover, $\theta(1)\sim \textbf{1}$ and so according to Proposition \ref{non trivial cocycle for cores} the group $G$ is abelian. Hence $Q$ is abelian and we can conclude by Theorem \ref{involutory simply}. 
\end{proof}

Let $Q$ be a quandle and $H_n=\langle L_x^n L_y^{-n},\, x,y\in Q\rangle$ and $\mathcal{O}_n=\mathcal{O}_{H_n}$. The subgroup $H_n$ is normal in $\lmlt(Q)$ and so $\mathcal{O}_n$ is a congruence. 
\begin{corollary}\label{cor on involutory}
    Let $Q$ be a finite simply connected latin quandle.
    \begin{itemize}
        \item[(i)] If $Q$ is nilpotent then $Q/\mathcal{O}_2\cong \aff(\Z_m,-1)$ for some odd integer $m\in \mathbb{N}$.

        \item[(ii)] If $Q$ is involutory then $Q/\gamma_Q\cong \aff(\Z_m,-1)$ for some odd integer $m\in \mathbb{N}$ and $\gamma_1(\dis(Q))=\gamma_2(\dis(Q))$. 
    \end{itemize}
\end{corollary}

\begin{proof}
Let $G=\dis(Q)$. Recall that the quandle $Q/\gamma_Q$ is affine over $\dis(Q/\gamma_Q)\cong G/\gamma_1(G)$, see Proposition \ref{gamma for quandles}.

    (i) The quandle $Q/\mathcal{O}_2$ is involutory, nilpotent and simply connected. So we can apply Theorem \ref{involutory simply}.
    
  (ii)  The quandle $Q/\gamma_Q\cong \aff(G/\gamma_1(G),f)$ is abelian and simply connected. So we can apply Theorem \ref{involutory simply} and so $G/\gamma_1(G)$ is cyclic. According to \cite[Proposition 9.2.5]{Sims} we have that $\gamma_1(G)/\gamma_2(G)$ is generated by $\setof{[x,y]}{x,y\in X}$ provided that $G/\gamma_1(G)$ is generated by $\setof{x\gamma_1(G)}{x\in X}$. Since $\dis(Q/\gamma_Q)\cong G/\gamma_1(G)$ is cyclic then $|X|$ has size $1$ and accordingly $\gamma_1(G)/\gamma_2(G)=1$. 
\end{proof}
Theorem \ref{involutory simply} and Corollary \ref{cor on involutory}(ii) do not extend to non-latin quandles. For instance for the quandle $Q=${\tt SmallQuandle(27,1)} in the \cite{RIG} database of connected quandles is a nilpotent involutory simply connected quandle, but it is not affine and $Q/\gamma_Q=\aff(\Z_p^2,-1)$.




\section{Appendix: groups of size \texorpdfstring{$p^4$}{p^4}}\label{appendix}

We denote by $G_i$ for $i=1,\ldots, 15$ the groups of size $p^4$ following \cite{tedesco}. In this section we show the power-commutator presentation of the non-abelian groups in such a list and the description of their automorphisms.

\begin{itemize}
    \item $ G_3 = \langle a_1, a_2, a_3, a_4 \mid [a_2, a_1] = a_3, a_1^p = a_4 \rangle.$
    In particular $\langle a_3\rangle=\gamma_1(G_3)\leq  Z(G_3)$ and $G_3/\gamma_1(G_3)\cong \Z_{p^2}\times \Z_p$.
    The automorphisms of $G_3$ are of the form:
\begin{align}\label{G_3_aut}
 \varphi(u_1, u_2, u_3, u_4, v_2, v_3, v_4) =\begin{cases}
         a_1 \mapsto a_1^{u_1} a_2^{u_2} a_3^{u_3} a_4^{u_4}, \\
    a_2 \mapsto a_2^{v_2} a_3^{v_3} a_4^{v_4}, \\
    a_3 \mapsto a_3^{u_1 v_2}, \\
    a_4 \mapsto a_4^{u_1},
 \end{cases}
\end{align}
where $u_1, u_2, u_3, u_4, v_2, v_3, v_4 \in \{0, \dots, p-1\}$ and $u_1 \neq 0$, $v_2 \neq 0$.

\item $G_4 = \langle a_1, a_2, a_3, a_4 \mid [a_2, a_1] = a_4, a_1^p = a_3, a_2^p = a_4 \rangle$. In particular $\langle a_3\rangle=\gamma_1(G_4)\leq Z(G_4)$ and $G_4/\gamma_1(G_4)\cong \Z_{p^2}\times \Z_p$.
 The automorphisms of $G_4$ are of the form:
 \begin{align}\label{G_4_aut}
     \varphi(u_2, u_3, u_4, v_2, v_3, v_4) =\begin{cases}
    a_1 \mapsto a_1 a_2^{u_2} a_3^{u_3} a_4^{u_4}, \\
    a_2 \mapsto a_2^{v_2} a_3^{v_3} a_4^{v_4}, \\
    a_3 \mapsto a_3^{v_2} , \\
    a_4 \mapsto a_3^{v_2}a_4,
    \end{cases}
\end{align}
where $u_2, u_3, u_4, v_2, v_3, v_4 \in \{0, \dots, p-1\}$ and $v_2 \neq 0$. 

\item  $G_6 = \langle a_1, a_2, a_3, a_4 \mid [a_2, a_1] = a_4, a_1^p = a_3, a_3^p = a_4\rangle$. In particular $\langle a_4\rangle=\gamma_1(G_6)\le Z(G_7)$ and $G_6/\gamma_1(G_6)\cong \Z_{p^2}\rtimes \Z_p$. The automorphisms of $G_6$ are of the form:
\begin{align}\label{G_6_aut}
  \varphi(u_1, u_2, u_3, u_4, v_4) =\begin{cases}  a_1 \mapsto a_1^{u_1} a_2^{u_2} a_3^{u_3} a_4^{u_4}, \\
    a_2 \mapsto a_2 a_4^{v_4}, \\
    a_3 \mapsto a_3^{u_1} a_4^{u_3}, \\
    a_4 \mapsto a_4^{u_1},
    \end{cases}
\end{align}
where $u_1, u_2, u_3, u_4, v_4 \in \{0, \dots, p-1\}$ and $v_2 \neq 0$.

\item $G_7 = \langle a_1, a_2, a_3, a_4 \mid [a_2, a_1] = a_3, [a_3, a_1] = a_4 \rangle$. In particular $\langle a_4\rangle= Z(G_7)\leq \gamma_1(G_7)=\langle a_3,a_4\rangle$ and $G_7/\langle a_4\rangle \cong \Z_p^2\rtimes \Z_p$. The automorphisms of $G_7$ are of the form:

\begin{align}\label{G_7_aut}
    \varphi(u_1, u_2, u_3, u_4, v_2, v_3, v_4) :
\begin{cases}
a_1 \mapsto a^{u_1}_1 a^{u_2}_2 a^{u_3}_3 a^{u_4}_4 \\
a_2 \mapsto a^{v_2}_2 a^{v_3}_3 a^{v_4}_4 \\
a_3 \mapsto a^{u_1v_2}_3 a_4^{u_1v_3 + u_1 v_3\frac{(u_1-1)}{2}}  \\
a_4 \mapsto a^{u_1^2v_2}_4
\end{cases}
\end{align}

where $u_1, u_2, u_3, u_4, v_2,v_3,v_4 \in \{0, \dots, p-1\}$ and $u_1v_2 \neq 0$.

\item $G_8 = \langle a_1, a_2, a_3, a_4 \mid [a_2, a_1] = a_3, [a_3, a_1] = a_4, a_1^p = a_4 \rangle$. In particular, $\langle a_4\rangle= Z(G_8)\leq \gamma_1(G_8)=\langle a_3,a_4\rangle$ and $G_8/\langle a_4\rangle \cong \Z_p^2\rtimes \Z_p$.
The automorphisms of $G_8$ are of the form:
\begin{align}\label{G_8_aut}
     \varphi(u_1, u_2, u_3, u_4, v_2, v_3, v_4) =\begin{cases}
           a_1 \mapsto a_1^{u_1} a_2^{u_2} a_3^{u_3} a_4^{u_4}, \\
     a_2 \mapsto a_2^{-u_1} a_3^{v_3} a_4^{v_4}, \\
     a_3 \mapsto a_3 a_4^{u_1 v_3 +\frac{u_1-1}{2}}, \\
     a_4 \mapsto a_4^{u_1},
         \end{cases} 
\end{align}
where \( u_1, u_2, u_3, u_4, v_3, v_4 \in \{0, \dots, p - 1\} \) and \( u_1 \neq 0 \).

\item $G_9 = \langle a_1, a_2, a_3, a_4 \mid [a_2, a_1] = a_3, [a_3, a_1] = a_4, a_2^p = a_4 \rangle$. In particular, $\langle a_4\rangle=Z(G_9)\leq \gamma_1(G_9)=\langle a_3,a_4\rangle$ and $G_9/\langle a_4\rangle \cong \Z_p^2\rtimes \Z_p$.
The automorphisms of $G_9$ are of the form:
\begin{align}\label{G_9_aut}
\varphi(u_3, u_4, v_2, v_3, v_4) = &
\begin{cases}
    a_1 \mapsto a_1^{\pm1} a_3^{u_3} a_4^{u_4}, \\
    a_2 \mapsto a_2^{v_2} a_3^{v_3} a_4^{v_4}, \\
    a_3 \mapsto a_3^{\pm v_2} a_4^{\pm v_3}, \\
    a_4 \mapsto a_4^{ v_2},
\end{cases}
\end{align}
where $u_3, u_4, v_2, v_3, v_4 \in \{0, \dots, p-1\}$ and $u_1 \neq 0$.
\item $G_{10} = \langle a_1, a_2, a_3, a_4 \mid [a_2, a_1] = a_3, [a_3, a_1] = a_4, a_2^p = a_4^w \rangle$. In particular, $\langle a_4\rangle=Z(G_{10})\leq \gamma_1(G_{10})=\langle a_3,a_4\rangle$, $G_{10}/\langle a_4\rangle \cong \Z_p^2\rtimes \Z_p$ and $G_{10}/\gamma_1(G_{10})\cong \Z_p^2$.
The automorphisms of $G_{10}$ are of the form:
\begin{align}\label{G_10_aut}
    \varphi(u_3, u_4, v_2, v_3, v_4)=\begin{cases}
        a_1 \mapsto a_1^{\pm 1} a_3^{u_3} a_4^{u_4}, \\
     a_2 \mapsto a_2^{v_2} a_3^{v_3} a_4^{v_4}, \\
     a_3 \mapsto a_3^{\pm v_2} a_4^{\pm v_3}, \\
         a_4 \mapsto a_4^{ v_2},
        \end{cases}
\end{align}
where $u_3, u_4, v_2, v_3, v_4 \in \{0, \dots, p - 1\}$ and $u_1 \neq 0$.

\item $G_{12} = \langle a_1, a_2, a_3, a_4 \mid [a_2, a_1] = a_4 \rangle$. In particular, $\langle a_4\rangle=\gamma_1(G_{12})\leq Z(G_{12})$ and $G_{12}/\gamma_1(G_{12}) \cong \Z_p^3$. 
The automorphisms of $G_{12}$ are of the form:
\begin{align}\label{G_12_aut}
   \varphi(u_1,u_2,u_3, u_4, v_1,v_2, v_3, v_4,w_3,w_4)=\begin{cases} a_1 \mapsto a_1^{u_1} a_2^{u_2} a_3^{u_3} a_4^{u_4}, \\
    a_2 \mapsto a_1^{v_1} a_2^{v_2} a_3^{v_3} a_4^{v_4}, \\
    a_3 \mapsto a_3^{w_3} a_4^{w_4}, \\
    a_4 \mapsto a_4^{v_2 u_1 - v_1 u_2},
    \end{cases}
\end{align}
where $u_1, u_2, u_3, u_4, v_1, v_2, v_3, v_4, w_3, w_4 \in \{0, \dots, p-1\}$, with $u_1 v_2 - v_1 u_2 \neq 0$ and $w_3 \neq 0$.
\item $G_{13} = \langle a_1, a_2, a_3, a_4 \mid [a_2, a_1] = a_4, a_1^p = a_4 \rangle$. In particular, $\langle a_4\rangle=\gamma_1(G_{13})\leq Z(G_{13})$ and $G_{13}/\gamma_1(G_{13}) \cong \Z_p^3$. The automorphisms of $G_{13}$ are of the form:
\begin{align}\label{G_13_aut}
\varphi(u_1, u_2, u_3, u_4, v_2, v_3, v_4, w_3, w_4)=\begin{cases}     
      a_1 \mapsto a_1^{u_1} a_2^{u_2} a_3^{u_3} a_4^{u_4}, \\
      a_2 \mapsto a_2^{v_2} a_3^{v_3} a_4^{v_4}, \\
      a_3 \mapsto a_3^{w_3} a_4^{w_4}, \\
      a_4 \mapsto a_4^{u_1 v_2},
    \end{cases}
\end{align}
where $u_1, u_2, u_3, u_4, v_2, v_3, v_4, w_3, w_4 \in \{0, \dots, p-1\}$ and $u_1 \neq 0$, $v_2 \neq 0$.
\item $G_{14} = \langle a_1, a_2, a_3, a_4 \mid [a_2, a_1] = a_4, a_3^p = a_4 \rangle$. In particular, $\langle a_4\rangle=\gamma_1(G_{14})\leq Z(G_{14})$ and $G_{14}/\gamma_1(G_{14}) \cong \Z_p^3$. 
The automorphisms of $G_{14}$ are of the form:
\begin{align}\label{G_14_aut}
   \varphi(u_1, u_2, u_4, v_1,v_2, v_4, w_4)=\begin{cases}
      a_1 \mapsto a_1^{u_1} a_2^{u_2} a_4^{u_4}, \\
      a_2 \mapsto a_1^{v_1} a_2^{v_2} a_4^{v_4}, \\
      a_3 \mapsto a_3^{u_1 v_2 - v_1 u_2} a_4^{w_4}, \\
      a_4 \mapsto a_4^{u_1 v_2 - v_1 u_2},
    \end{cases}
\end{align}
where $u_1, u_2, u_4, v_2, v_4,w_4 \in \{0, \dots, p - 1\}$ and $u_1 v_2 - v_1 u_2 \neq 0$. 
\end{itemize}


Let $G$ be a finite $p$-group and $f\in \aut{G}$. We focus on the pairs $(G,f)$ such that 
\begin{equation}\label{good}
Fix(f_{\Phi(G))})=1,\quad Fix(f)\leq Z(G)\cap \gamma_1(G),\quad |Fix(f)|=p.    
\end{equation}

\begin{lemma}\label{which groups}
    Let $G_i$ be a non-abelian group of size $p^4$ and $f\in \aut{G_i}$. If $(G_i,f)$ satisfies \eqref{good} then $i\in \{3,7,12,13\}$.
\end{lemma}

\begin{proof}
We employ a case by case discussion:

    \begin{enumerate}
        \item[(i)] $i=4,6,8$: according to \eqref{G_4_aut}, \eqref{G_6_aut} and \eqref{G_8_aut} $f_{\Phi(G_i)}$ is a matrix with an eigenvalue equals to $1$. Hence $Fix(f_{\Phi(G_i)})\neq 1$. 
        

        
        \item[(iv)] $i=9,10$: note that $\gamma_1(G_i)\cap Z(G_i)=\langle a_4\rangle$, so according to \eqref{G_9_aut} and \eqref{G_10_aut} $v_2=1$. Therefore, $f_{\Phi(G_i)}$ has an eigenvalue equal to $1$. So $Fix(f_{\Phi(G_i)})\neq 1$.

                \item[(v)] $i=14$ note that $\gamma_1(G_{14})\cap Z(G_i)=\langle a_4\rangle$, so according to \eqref{G_14_aut} $u_1 v_2-v_1 u_2=1$. Therefore, $f_{\Phi(G_{14})}$ has an eigenvalue equal to $1$. So $Fix(f_{\Phi(G_{14})})\neq 1$.\qedhere
    \end{enumerate}
\end{proof}

The following statement can be easily proved by looking at the description of the automorphisms of the groups $G_i$ for $i=3,7,12,13$ provided at the beginning of the section.
\begin{lemma}\label{what auto?}
    Let $i=3,7,12,13$ and $f\in \aut{G_i}$. Then: \begin{itemize}
        \item[(i)] $(G_3,f)$ satisfies \eqref{good} if and only if $u_1 v_2=1 \pmod{p}$, $ u_1,v_2\neq 1 \pmod{p}$. 
    
        \item[(ii)] $(G_7,f)$ satisfies \eqref{good} if and only if  $u_1^{2} v_2=1\pmod{p}$, $u_1,v_2\neq 1\pmod{p}$.
    
        \item[(iii)] $(G_{12},f)$ satisfies \eqref{good} if and only if $u_1 v_2-v_1 u_2=1\pmod{p}$, $u_1+v_2\neq 2$. 
        
        \item[(iv)] $(G_{13},f)$ satisfies \eqref{good} if and only if $ u_1 v_2 =1\pmod{p}$, $u_1,v_2\neq 1\pmod{p}$.
    \end{itemize}
\end{lemma}

\bibliographystyle{amsalpha}
\bibliography{references}

\end{document}